\documentclass[a4paper,10pt]{article}
\usepackage{mathtext}
\usepackage[T1,T2A]{fontenc}
\usepackage[cp1251]{inputenc}
\usepackage[english]{babel}
\usepackage{amsmath}
\usepackage{amsfonts}
\usepackage{amssymb}
\usepackage{mathrsfs}
\usepackage{amsthm}
\usepackage{titling}
\usepackage{multicol}
\usepackage{graphicx}
\usepackage{wrapfig}
\usepackage{appendix}
\usepackage{stmaryrd}
\usepackage{relsize}

\usepackage{color}
\usepackage{euscript}

\usepackage{cite}
\usepackage{mathtools}
\usepackage{enumerate}

\newcommand{\one}{{\bf 1}}

\newcommand{\N}{\mathbb{N}}
\newcommand{\Z}{\mathbb{Z}}
\newcommand{\El}{\mathfrak{L}}
\newcommand{\BEl}{\Sigma_{\scriptscriptstyle\mathrm{bl}}}
\newcommand{\Arc}{\mathfrak{A}}

\newcommand{\arc}{\mathfrak{a}}
\newcommand{\arcr}{\mathfrak{a}^{\mathrm{r}}}
\newcommand{\arcl}{\mathfrak{a}^{\mathrm{l}}}
\newcommand{\BMO}{\mathrm{BMO}}
\newcommand{\eps}{\varepsilon}
\newcommand{\av}[2]{\langle {#1}\rangle_{{}_{#2}}}
\renewcommand{\le}{\leqslant}

\renewcommand{\leq}{\leqslant}
\renewcommand{\geq}{\geqslant}

\newcommand{\FixedBoundary}{\partial_{\mathrm{fixed}}}
\newcommand{\FreeBoundary}{\partial_{\mathrm{free}}}

\newcommand{\ur}{u_{\scriptscriptstyle\mathrm{R}}}
\newcommand{\ul}{u_{\scriptscriptstyle\mathrm{L}}}
\newcommand{\vr}{v_{\scriptscriptstyle\mathrm{R}}}
\newcommand{\vl}{v_{\scriptscriptstyle\mathrm{L}}}
\newcommand{\Rt}{\Omega_{\scriptscriptstyle\mathrm{R}}}
\newcommand{\Lt}{\Omega_{\scriptscriptstyle\mathrm{L}}}
\newcommand{\SRt}{\Sigma_{\scriptscriptstyle\mathrm{R}}}
\newcommand{\SLt}{\Sigma_{\scriptscriptstyle\mathrm{L}}}
\newcommand{\Ch}{\Omega_{\scriptscriptstyle\mathrm{Ch}}}

\newcommand{\Ang}{\Omega_{\scriptscriptstyle\mathrm{ang}}}
\newcommand{\Sang}{\Sigma_{\scriptscriptstyle\mathrm{Sq}}}
\newcommand{\RCorn}{\Sigma_{\scriptscriptstyle\mathrm{Corn,R}}}
\newcommand{\LCorn}{\Sigma_{\scriptscriptstyle\mathrm{Corn,L}}}
\newcommand{\NT}{\mathrm{ST}}

\newcommand{\Mrt}{m_{\scriptscriptstyle\mathrm{R}}}
\newcommand{\Mlt}{m_{\scriptscriptstyle\mathrm{L}}}

\newcommand{\F}{\mathcal{F}}

\newcommand{\eq}[1]{\begin{equation}{#1}\end{equation}}
\newcommand{\mlt}[1]{\begin{multline}{#1}\end{multline}}
\newcommand{\alg}[1]{\begin{align}{#1}\end{align}}

\newcommand{\set}[2]{\{{#1}\mid{#2}\}}
\newcommand{\Set}[2]{\Big\{{#1}\,\Big|\;{#2}\Big\}}

\newcommand{\SSet}[2]{\Bigg\{{#1}\,\Bigg|\;{#2}\Bigg\}}

\newcommand{\Eeqref}[1]{\stackrel{\scriptscriptstyle{\eqref{#1}}}{=}}
\newcommand{\Leqref}[1]{\stackrel{\scriptscriptstyle{\eqref{#1}}}{\leq}}

\newcommand{\Eref}[1]{\stackrel{#1}{=}}

\newcommand{\BG}{\mathfrak{B}}

\newcommand{\E}{\mathbb{E}}
\newcommand{\R}{\mathbb{R}}
\DeclareMathOperator{\cl}{cl}

\newcommand{\UB}{\mathbb{U}}
\newcommand{\VB}{\mathbb{V}}
\newcommand{\VG}{\mathfrak{V}}
\newcommand{\B}{\mathbb{B}}

\newcommand{\Dr}{D_{\mathrm{R}}}
\newcommand{\Dl}{D_{\mathrm{L}}}

\newcommand{\SFir}{\Sigma_{\mathrm{hrb}}}

\newcommand{\Concave}{\mathfrak{C}}
\newcommand{\Diagonal}{\mathfrak{D}}

\newcommand{\Strip}{\Sigma}

\newcommand{\bz}{\tilde{z}}
\newcommand{\by}{\tilde{y}}

\newcommand{\RTroll}{\Omega_{\mathrm{tr},\scriptscriptstyle\mathrm{R}}}
\newcommand{\LTroll}{\Omega_{\mathrm{tr},\scriptscriptstyle\mathrm{L}}}

\newtheorem{Le}{Lemma}[subsection]
\newtheorem{Def}[Le]{Definition}
\newtheorem{St}[Le]{Proposition}
\newtheorem{Th}[Le]{Theorem}
\newtheorem{Cor}[Le]{Corollary}
\newtheorem{Rem}[Le]{Remark}
\newtheorem{Fact}[Le]{Fact}
\newtheorem{Conj}[Le]{Conjecture}
\newtheorem{Cond}[Le]{Condition}
\newtheorem{Que}[Le]{Question}
\numberwithin{equation}{section}

\numberwithin{equation}{subsection}

\begin{document}
\author{Dmitriy Stolyarov, Vasily Vasyunin, Pavel Zatitskii}
\title{Martingale transforms of bounded random variables {and indicator functions of events}\thanks{Supported by the Russian Science Foundation grant no 19-71-10023.}}

\maketitle
\begin{abstract}
We provide sharp estimates for the distribution function of a martingale transform of the {indicator} function of an event. They are formulated in terms of Burkholder functions, which are reduced to the already known Bellman functions for extremal problems on~$\BMO$. {The reduction implicitly uses an unexpected phenomenon of automatic concavity for those Bellman functions: their concavity in some directions implies concavity with respect to other directions. A similar question for a martingale transform of a bounded random variable is also considered.}
\end{abstract}

\section{Introduction}\label{S1}
\subsection{Stein--Weiss formula and martingale transforms}\label{s11}
The classical Stein--Weiss formula from~\cite{SteinWeiss1959} says that the distribution function of the Hilbert transform of the {indicator} function of a set depends on its Lebesgue measure only. More precisely,
\eq{
\Big|\Set{x\in \R}{\big|{\rm H}[\one_E](x)\big| \geq t}\Big| = \frac{2|E|}{\sinh \pi t}.
}
The absolute value of a set means its Lebesgue measure. For the case where~$E$ is a finite union of intervals, this result is, in fact, equivalent to Boole's identity for rational functions from~\cite{Boole1857}; see~\cite{CLM2010} for explanations. Results in Harmonic Analysis quite often have counterparts in Probability Theory. Our aim in this paper is to try to find probabilistic analogs of the Stein--Weiss formula. A natural choice would be to replace the Hilbert transform with a suitable martingale transform.

Let~$(S,\mathcal{F},P)$ be the standard atomless probability space. Let~$\{\F_n\}_n$ be a filtration, i.\,e., an increasing sequence of~$\sigma$-fields that generate~$\F$. For simplicity, we assume that each~$\F_n$ is generated by at most countable number of atoms and that~$\F_0$ is the trivial~$\sigma$-field. Let~$\varphi = \{\varphi_n\}_n$ be a martingale adapted to~$\{\F_n\}_n$, let~$\psi = \{\psi_n\}_n$ be its martingale transform by a bounded predictable sequence~$\{\alpha_n\}_n$; we assume~$\|\alpha_n\|_{L_\infty}\leq 1$ for any~$n\in \N\cup\{0\}$. The predictability of the transforming sequence means the variable~$\alpha_n$ is~$\F_{n-1}$-measurable. The transformation rule reads as~$d\psi_n = \alpha_n d\varphi_n$. Here
\eq{
d\varphi_n = \varphi_{n} - \varphi_{n-1} \quad \text{and}\quad d\psi_n = \psi_{n} - \psi_{n-1}
}
are the martingale difference sequences. By the general martingale theory (for which we broadly refer to Chapter~VII in~\cite{Shiryaev2019}),~$\psi$ is also a martingale and the mapping~$\varphi \mapsto \psi$ is continuous in~$L_p$ for~$p \in (1,\infty)$. If~$\varphi$ is a uniformly bounded martingale, by Doob's theorem it has the limit value~$\varphi_\infty$. {By} the above,~$\psi$ also has the limit value~$\psi_\infty \in L_p$,~$p < \infty$. 

Let~$f$ be a fixed function. Later we will impose smoothness and regularity requirements on this function, for now we assume it is `good'. We address two questions:

\medskip

\centerline{\bf Question 1: What are the sharp estimates of~$\E f(\psi_\infty)$}

\centerline{\bf  provided~$\varphi$ is a uniformly bounded martingale?}

\medskip

\centerline{\bf Question 2: What are the sharp estimates of~$\E f(\psi_\infty)$}

\centerline{\bf  provided~$\varphi_\infty$ is {the indicator of} an event?}

\medskip

These questions need clarification. The random variable~$\psi_\infty$ is defined modulo the constant~$\psi_0$. Either we set~$\psi_0=0$, or need to track the dependence of our answer on this constant.

We will provide a complete answer to the second question, thus obtaining a probabilistic version of the Stein--Weiss formula. To do this, we will need to restrict the class of martingale transforms we will be working with. For the second question, we will consider the sequences~$\{\alpha_n\}_n$ such that the random variables~$\alpha_n$ attain values~$\pm 1$ only. This will give the problem sufficient `rigidity' that enables a certain analog of Boole's formula. {In a sense,~$\{\alpha_n\}_n$ is the probabilistic substitute for a Fourier multiplier; the sequences with~$\alpha_n = \pm 1$ resemble the Hilbert transform, which is the Fourier multiplier with the symbol~$i\one_{(0,\infty)} - i\one_{(-\infty,0)}$.} It will be also more convenient to work with the case where~$\varphi_\infty$ attains the values~$1$ or~$-1$ rather than~$0$ and~$1$. The two cases are equivalent at the cost of replacing~$f$ by the function~$t\mapsto f((t+1)/2)$. As for the first question, we will provide only a partial answer.

It will be also convenient to work with simple martingales, i.\,e., those that do not develop after some time and attain a finite number of values; with this assumption, we may also think that {each}~$\F_n$ {is} generated by a finite number of sets. If we obtain a bound for~$\E f(\psi_\infty)$ for simple martingales, then the same bound may be extended to the case of an arbitrary bounded martingale~$\varphi$ by a limit argument. From now on, we assume~$\varphi$ is a simple martingale.

{We consider Theorems~\ref{TheoremForSubordination} and~\ref{TheoremForDiagonal} as the main results of the paper. The construction of the function~$\VB$ and Fact~\ref{Auto} also seem important to us, though we do not see direct applications now. 

We wish to thank Paata Ivanisvili and Alexander Logunov with whom we discussed the  topics around the present paper long time ago. We are grateful to Nicolai Krylov and Xu-Jia Wang for providing references. We also wish to thank Ilya Zlotnikov for reading the manuscript.
}

\subsection{Burkholder's functions}\label{s12}
As we have already noted, the answer to the questions should depend on~$\psi_0$. It is also convenient to fix~$\varphi_0$ and introduce the Burkholder functions
\alg{
\label{UDef}\UB(x_1,x_2) &= \sup\Set{\E f(\psi_\infty)}{\psi_0 = x_1,\ \varphi_0 = x_2, \ \|\varphi_{\infty}\|_{L_\infty} \leq 1};\\
\label{VDef}\VB(x_1,x_2) &= \sup\Set{\E f(\psi_\infty)}{\psi_0 = x_1,\ \varphi_0 = x_2,\ \forall n\quad |\alpha_n| = 1,\ \ |\varphi_{\infty}| = 1\ \text{a.\,s.}}.
}
The two main questions may be then rephrased as: {\bf are there good analytic expressions for~$\UB$ and~$\VB$?} It was first observed by Burkholder in~\cite{Burkholder1984} that functions of such type  are solutions to certain boundary value problems; this information gives hope that~$\UB$ and~$\VB$ may be expressed analytically. The origins of the method may be traced back to the classical moment method (see~\cite{Kemperman1973} and~\cite{KN1977} for {this method}).

The functions~$\UB$ and~$\VB$ are defined on the entire plane~$\R^2$. The points~$x = (x_1,x_2)$ with~$|x_2| > 1$ are uninteresting: if~$|x_2| > 1$, then there are no~$\varphi$ such that~$\varphi_0 = x_2$ and~$\|\varphi_\infty\|_{L_\infty} \leq 1$. Therefore,
\eq{
\UB(x) = \VB(x) = -\infty,\qquad |x_2| > 1.
}
From now on we disregard these void cases and treat~$\UB$ and~$\VB$ as functions on
\eq{
\Strip = \Set{x\in \R^2}{|x_2| \leq 1}.
}
We may always plug constant martingales~$\varphi = \varphi_0=x_2$ and~$\psi = \psi_0=x_1$ into the definition of~$\UB$. This leads us to the obstacle condition
\eq{
\UB(x) \geq f(x_1),\qquad x\in \Sigma.
}
The function~$\VB$, in its turn, satisfies the boundary conditions
\eq{
\VB(x_1, \pm 1) = f(x_1),\qquad x_1 \in \R,
}
since if~$\varphi_0 = \pm 1$ and~$\|\varphi_\infty\|_{L_\infty} \leq 1$, then both~$\varphi$ and~$\psi$ are constant martingales:~$\varphi_\infty = \varphi_0=\pm 1$ and~$\psi_\infty = \psi_0=x_1$.

\begin{Def}
Let~$\beta \in \R$. We say that a function~$G \colon \Strip \to \R$ is concave along the direction~$(\beta, 1),$ provided the function
\eq{
t\mapsto G(x_1 + \beta t, x_2 + t),
}
is concave on its domain for any~$x\in \Strip$. We say that a function~$G \colon \Strip \to \R$ is diagonally concave if it is concave along the directions~$(1,1)$ and~$(-1,1)$.
\end{Def}
\begin{Def}
Let~$f\colon \R\to \R$ be a function. Define the classes of functions
\alg{
\Concave[f] &= \Set{G\colon \Sigma \to \R}{G \ \text{is concave along $(\beta,1),$~$\beta \in [-1,1],$ and}\ G(x) \geq f(x_1),\ x\in \Sigma};\\
\label{DiagonalClass}\Diagonal[f] &= \Set{G\colon \Sigma \to \R}{G \ \text{is diagonally concave and}\ G(x_1,\pm 1) \geq f(x_1)\ \text{for all}\ x_1\in \R}.
}
\end{Def} 
We note that if~$\{G_\alpha\}_\alpha$ is a family of functions on~$\Sigma$ that are concave along some direction, then the function~$G = \inf_\alpha G_\alpha$ is also concave along the same direction, provided it attains finite values. This allows to define the pointwise minimal functions in the families~$\Concave[f]$ and~$\Diagonal[f]$. {The two theorems below may be traced to Burkholder.}
\begin{Th}
For any function~$f\colon \R\to \R,$
\eq{
\UB(x) = \inf\Set{G(x)}{G\in \Concave[f]},\qquad x\in \Sigma.
}
\end{Th}
\begin{Th}\label{BurkholderV}
For any function~$f\colon \R\to \R,$
\eq{
\VB(x) = \inf\Set{G(x)}{G\in \Diagonal[f]},\qquad x\in \Sigma.
}
\end{Th}
Note that by the minimality principle above,~$\UB \in \Concave[f]$ and~$\VB \in \Diagonal[f]$.
For the proofs of the two theorems above, see Chapter~$2$ in~\cite{Osekowski2012}.

\subsection{Bellman functions for $\BMO$}\label{s13}
Define the~$\BMO$ norm of a {summable} function~$\zeta\colon [0,1]\to \R$ by the rule
\eq{\label{BMONorm}
\|\zeta\|_{\BMO}^2 = \sup\Set{\av{|\zeta - \av{\zeta}{I}|^2}{I}}{ I\ \text{is a subinterval of}\ [0,1]}.
}
Here~$\av{\zeta}{E}$ means the average of a function~$\zeta$ over a set~$E\subset \R$, that is,~$|E|^{-1}\int_E \zeta(t)\,dt$
. 
We refer the reader to {Chapter IV} in~\cite{Stein1993} for more information on the space~$\BMO$. We only mention that usually the seminorm~\eqref{BMONorm} is defined via~$L_1$ norm instead of the~$L_2$ norm. The two definitions are equivalent, they define the same space of functions. The sharp constants in the corresponding inequalities were investigated in~\cite{SlavinVasyunin2012}.

Define the Bellman function
\eq{\label{BellmanFunctionInTheStrip}
\B_\eps(y_1,y_2) = \sup\Set{\int_0^1 f(\zeta(t))\,dt}{\av{\zeta}{[0,1]} = y_1,\ \av{\zeta^2}{[0,1]} = y_2,\ \|\zeta\|_{\BMO} \leq \eps},\qquad y\in \R^2,
}
where~$\eps > 0$ is a parameter. Though we will only need the case~$\eps = 1$, we prefer to work with arbitrary~$\eps$ for a while.

Similar to the Burkholder functions above, the natural domain for~$\B_\eps$ is the parabolic strip
\eq{\label{ParabolicStrip}
\Omega_\eps = \Set{y\in\R^2}{y_1^2 \leq y_2 \leq y_1^2+\eps^2}.
}
The function~$\B_\eps$ satisfies the boundary condition
\eq{\label{BCForBMO}
\B_\eps(y_1,y_1^2) = f(y_1),\qquad y_1 \in \R,
}
on the lower boundary of the strip (we call this part of the boundary the \emph{fixed} boundary; the other part of the boundary,~$y_2 = y_1^2 + \eps^2$\!, is called the \emph{free} boundary\footnote{{This has nothing common with the free boundary problems. This part of the boundary is called free since we do not define the boundary data on it.}}). 
The function~$\B_\eps$ also solves a certain boundary value problem. Let~$\omega \subset \R^d$ be an arbitrary set. We say that a function~$G\colon \omega \to \R$ is locally concave, provided its restriction to any convex set~$C\subset \omega$ is concave{; sometimes locally concave functions are called concave functions on non-convex domains.} We define
\eq{\label{SetLambda}
\Lambda[f,\eps] = \Set{G \colon \Omega_\eps\to \R}{G\ \text{is locally concave and }\ G(y_1,y_1^2) = f(y_1),\ y_1\in \R}.
} 
\begin{Th}[Main Theorem in \cite{StolyarovZatitskiy2016}]\label{SZTheorem}
Assume~$f$ is locally bounded and uniformly bounded from below. Then,
\eq{
\B_\eps(y) = \inf\Set{G(y)}{G\in \Lambda[f,\eps]},\qquad y\in \Omega_\eps.
}
\end{Th}
Theorem~\ref{SZTheorem} says nothing about analytic expressions for~$\B_\eps$. It supports the hope that such formulas might exist. The paper~\cite{ISVZ2018} provides an algorithm for finding such an expression. We will survey it in the next subsection. A good example explaining what we mean by an `expression' is given by the case~$f(t) = e^{\lambda t}$,~$\lambda \in (0,1)$, where
\eq{
\B_1(y) = \Big[\frac{\lambda}{1-\lambda}\Big(1 - \sqrt{y_1^2 - y_2 + 1}\Big) + 1\Big]\;e^{\lambda (y_1 - 1 + \sqrt{y_1^2 - y_2 + 1})}.
}
This case was considered in~\cite{SlavinVasyunin2011}. See Section~\ref{s34} below for more details. Now we are ready to formulate our main results. {Before that we simplify our notation: we will write~$\B$ instead of~$\B_1$; similar rule applies to~$\Omega_1$ and~$\Omega$.}

\begin{Th}\label{TheoremForSubordination}
Let~$f$ be locally bounded and uniformly bounded from below. Then\textup, 
\eq{\label{FirstMajorization}
\UB(x)\ \leq\!\! \sup\limits_{0\leq \delta \leq 1-x_2^2}\!\! \B(x_1,x_1^2+\delta),
}
for any~$x\in \Sigma$.
\end{Th}
The theorem above might be thought of as corollary of the considerations of~\cite{StolyarovZatitskiy2016}. We do not know when this estimate is sharp.
{Theorem~$6.1.2$ in~\cite{ISVZ2018} says that~$\B$ is finite if~$f$ is measurable and
\eq{\label{MaximalSum}
\sum\limits_{k\in \Z}e^{-|k|}\!\! \sup\limits_{[k-2,k+2]}\!\! |f| \;<\; \infty.
}
Thus, we obtain a corollary of Theorem~\ref{TheoremForSubordination}.
\begin{Cor}\label{BoundaryBehavior}
Assume~$f$ is a measurable function such that the sum~\eqref{MaximalSum} is finite. Then\textup,~$\E f(\psi_\infty)$ is uniformly bounded when~$\psi$ is a martingale transform of~$\varphi$ and~$\|\varphi\|_{L_\infty} \leq 1$.
\end{Cor}
The function~$f(t) = (1+t^2)^{-1}e^{|t|}$ fulfills~\eqref{MaximalSum} and the function~$f(t) = (1+|t|)^{-1}e^{|t|}$ does not. One may see that for the second function,~$\B = \infty$ (plug~$\zeta(t) = -\log t$ into~\eqref{BellmanFunctionInTheStrip}); Theorem~\ref{TheoremForDiagonal} below then says~$\VB$ and~$\UB$ are also infinite in this case).
}

The next theorem does not follow from abstract results of~\cite{StolyarovZatitskiy2016}. The full strength of the machinery from~\cite{ISVZ2018} will be needed for its proof.
\begin{Th}\label{TheoremForDiagonal}
Let~$f$ be  bounded from below and continuous. Then\textup,
\eq{\label{TheoremForDiagonalFormula}
\VB(x_1,0) = \B(x_1,x_1^2+1), \qquad x\in \R.
}
\end{Th}
{\begin{Rem}
We believe that the continuity of~$f$ in the theorem above is redundant. Since the problem is non-linear\textup, naive approximation arguments do not work.
\end{Rem}
}
Theorems~\ref{TheoremForSubordination} and~\ref{TheoremForDiagonal} are manifestations of a transference principle between the extremal problems~\eqref{BellmanFunctionInTheStrip} and~\eqref{UDef},~\eqref{VDef}; we describe this transference in the present paper. Similar transference was used in~\cite{SVZZ2022} to obtain sharp bounds for the terminate distribution of a martingale whose square function is uniformly bounded.

In fact, Theorem~\ref{TheoremForDiagonal} follows from a more general result that allows to compute~$\VB$ on the entire domain~$\Sigma$. In a sense, this computation is yet another main result of the paper. The answer is formulated in terms of~$\B$, however, the dependence is far from being simple. To formulate it, we need to recall the construction of~$\B$ from~\cite{ISVZ2018}. The forthcoming subsection is devoted to this sketch.


\subsection{Description of minimal locally concave functions}\label{s14}
Let us fix~$f$ and~$\eps$. Consider the function
\eq{\label{MinimalLocallyCOncave}
\BG_\eps(y) = \inf\Set{G(y)}{G\in \Lambda[f,\eps]},\qquad y \in \Omega_\eps.
}
As we know from Theorem~\ref{SZTheorem}, this function coincides with~$\B_\eps$. {Since} now we wish to focus on its geometric properties, {we denote it with another symbol}. Since this function is minimal, its graph is somehow flat. Using the Carath\'eodory theorem about convex hulls, one may prove the following lemma (see~\cite{RT1977},~\cite{CNS1986}, or~\cite{ISZ2015} for a similar reasoning). 
\begin{Le}
For any~$y$ from the interior of~$\Omega_\eps$ there exists an open line segment~$\ell(y)$ such that $y \in\ell(y) \subset \Omega_\eps$ and the restriction of~$\BG_\eps$ to~$\ell(y)$ is affine. Either~$\BG_\eps$ is affine in a neighborhood of~$y$ or the direction of~$\ell(y)$ is unique.
\end{Le}
For any~$y$ in the interior of $\Omega_\eps$ such that~$\BG_\eps$ is not affine in its neighborhood, we may consider the maximal by inclusion segment~$\ell_y$ such that~$\BG_\eps|_{\ell_y}$ is affine. These segments are called \emph{extremal}. We mention yet another folklore fact since it supports the intuition. We will not provide a proof since it is formally not needed.
\begin{Le}
If the function~$\BG_\eps$ is differentiable at~$y,$ then it is differentiable at every point~$z\in \ell_y$ and~$\nabla\BG_\eps (y) = \nabla \BG_\eps(z)$.
\end{Le}
If~$\BG_\eps$ is not differentiable at~$y,$ one may formulate a similar principle for superdifferentials. At the points where~$\BG_\eps$ is twice differentiable, it satisfies the homogeneous Monge--Amp\`ere equation
\eq{\label{MongeAmpere}
\mathrm{det}\begin{pmatrix}
\frac{\partial^2 \BG_\eps}{\partial y_1^2}&\frac{\partial^2 \BG_\eps}{\partial y_1\partial y_2}\\
\frac{\partial^2 \BG_\eps}{\partial y_1\partial y_2}&\frac{\partial^2 \BG_\eps}{\partial y_2^2}
\end{pmatrix}\;=\;
\frac{\partial^2 \BG_\eps}{\partial y_1^2}\cdot \frac{\partial^2 \BG_\eps}{\partial y_2^2} - \Big(\frac{\partial^2 \BG_\eps}{\partial y_1\partial y_2}\Big)^2\!=0.
}
We warn the reader that~$\BG_\eps$ is rarely~$C^2$-smooth on the entire~$\Omega_\eps$. See Subsection~\ref{s72} below concerning smoothness questions. 

One may show that the endpoints of~$\ell_y$ lie on the boundary of~$\Omega_\eps$. If both endpoints lie on the fixed boundary, then~$\ell_y$ is called a \emph{chord}. If one of the endpoints lie on the free boundary, then one may show that~$\ell_y$ is tangent to the free boundary. In such a case,~$\ell_y$ is called a \emph{tangent}. The tangents that lie on the right of their tangency points are called \emph{right} tangents, the ones that lie on the left are \emph{left} tangents.

If the function~$\BG_\eps$ coincides with an affine function~$A$ in a neighborhood of~$y$, then we may consider the set where~$\BG_\eps = A$ {and restrict our attention to its connectivity component that contains~$y$}. One may prove that the part of the boundary of this set that lies in the interior of~$\Omega_\eps$, is linear. Such sets are called \emph{linearity domains}. We see that~$\Omega_\eps$ is split into several linearity domains and families of extremal segments. Such a splitting is called a \emph{foliation}. Usually, one can restore the function~$\BG_\eps$ from the foliation using special formulas, and the problem of calculating~$\BG_\eps$ reduces to finding its foliation.

In general, there can be an infinite number of linearity domains in a foliation, and the structure might be quite involved. The paper~\cite{ISVZ2018} suggests regularity conditions on~$f$ under which there are a finite number of linearity domains.

\begin{Cond}\label{reg}
The function~$f$ is two times continuously differentiable\textup,\, $f''$\! is piecewise monotone
and has only finite number of monotonicity intervals.
\end{Cond}
Let~$w_{\eps}(t) = e^{-{|t|}/{\eps}}$ for~$\eps > 1$.
\begin{Cond}\label{sum}
The integral~$\int_{-\infty}^{\infty}w_{\eps}(t)\,df''(t)$ is absolutely convergent for some~$\eps > 1$.
\end{Cond}
Condition~\ref{sum} slightly differs  from $w_{\eps} \in L_1(df)$, which is necessary for the finiteness of the Bellman function $\B_\eps$; the necessity may be observed from plugging~$\zeta(t) = -\eps \log t$ into~\eqref{BellmanFunctionInTheStrip}. 

Note that Condition~\ref{reg} yields~$f'''$ is a signed measure in the sense of distributions. The \emph{essential roots} of $f'''$ will play a significant role in what follows. The essential roots may be defined as the largest by inclusion intervals on which~$f''' = 0$ and in whose neighborhood~$f'''$ is neither positive nor negative. If~$c$ and~$v$ are two intervals, we write~$c< v$, provided the right endpoint of~$c$ lies on the left of the left endpoint of~$v$. 
\begin{Def}\label{roots}
The essential roots of \,$f'''$\! are closed intervals \textup(which can be single points or rays\textup) 
$c_0, c_1, \ldots, c_n$ and $v_1, v_2, \ldots, v_n$ such that $c_0 < v_1 < c_1 < v_2 < \cdots < v_n < c_n$ and 
$(\cup_i v_i) \bigcup (\cup_i c_i)$ is the complement to the set of the growth points of \,$f''$\! 
\textup(i.\,e.\textup, the points, in a neighborhood of which \,$f''$\! either strictly increases or decreases\textup). 
The measure \,$f'''$ \textup`changes sign\textup' from \textup`$-$\textup' to \textup`$+$\textup' at $v_i,$ from 
\textup`$+$\textup' to \textup`$-$\textup' at~$c_i$. 
\end{Def} 

We make an agreement that if in a neighborhood of $-\infty$ we have $f''' < 0$, then $c_0 = -\infty$. Similarly, if in a neighborhood of $+\infty$ we have $f''' > 0$, then $c_n = \infty$. 

Under the conditions above, all linearity domains in our foliations will intersect the fixed boundary by at most finite number of arcs. The domains that do not meet the free boundary are called \emph{closed multicups}.

A linearity domain with a single point on the fixed boundary is called an \emph{angle} (see Fig.~\ref{AngleEnvelope} 
on p.~\pageref{AngleEnvelope}). A linearity domain with two points on the fixed boundary is called a \emph{trolleybus} 
(see Fig.~\ref{RightTrolleyFigure} on p.~\pageref{RightTrolleyFigure}) if it is bounded by equally oriented tangents 
and a \emph{birdie} if it is bounded by a right tangent on the left and a left tangent on the right. Other linearity 
domains will be introduced in Section~\ref{S5} below.

\subsection{Minimal diagonally concave functions}\label{s15}
Minimal diagonally concave functions are also somehow flat. Consider the function
\eq{\label{MinimalDiagonallyConcaveFunction}
\VG(x) = \inf\Set{G(x)}{G\in \Diagonal[f]},\qquad x\in \Sigma.
}
By Theorem~\ref{BurkholderV}, this function coincides with~$\VB$. Now we wish to study the geometric properties of~$\VG$. We leave the {folklore} theorem below without proof since we will not formally use it. Nevertheless, knowing it is true helps to guess the minimal diagonally concave functions.
\begin{Th}
Let~$f$ be a locally bounded function. Let~$\VG\colon \Sigma\to \R$ be the minimal diagonally concave function defined 
in~\eqref{MinimalDiagonallyConcaveFunction}\textup, let~$x\in \Sigma$. Then\textup, there exists a closed 
segment~$\ell$ with non-empty interior such that~$x\in \ell$\textup,\; $\VG|_\ell$ is affine\textup, and~$\ell$ is 
parallel to one of the directions~$(1,1)$ or~$(1,-1)$.
\end{Th}

In the study of locally concave functions, the affine functions play a special role: they are both concave and convex. In the case of diagonal concavity, affine functions have the same property. However, there is yet another function~$x_1^2 - x_2^2$ that is both diagonally concave and diagonally convex. Thus, for diagonally concave functions, we will consider \emph{domains of bilinearity}: these are the domains where~$\VG(x)$ equals~$c(x_1^2 - x_2^2) + L(x)$, where~$L$ is an affine function. Similar to the case of minimal locally concave functions, one may show that if we choose~$\ell$ to be maximal by inclusion such that~$\VG|_\ell$ is affine and~$\ell$ does not lie in a bilinearity domain, then at least one endpoint of~$\ell$ lies on the boundary of~$\Sigma$. Then,~$\VG$ solves the partial differential equation
\eq{\label{StrangeEquation}
\Big(\frac{\partial^2 \VG}{\partial x_1^2} + \frac{\partial^2 \VG}{\partial x_2^2}\Big)^2 \!= 4\, \Big(\frac{\partial^2 \VG}{\partial x_1\partial x_2}\Big)^2\!,
}
provided it is~$C^2$-smooth. See Subsection~\ref{s72} below for smoothness questions. Now we wish to survey the theory from~\cite{Novikov2022}. The cited paper works with biconcave functions, i.\,e., functions that are concave with respect to each of the coordinates individually. The theory is completely analogous since a biconcave function is a rotation of a diagonally concave function. 

\begin{Def}\label{DefinitionOfExtremal}
Let~$G\colon \Sigma\to \R$ be a diagonally concave function. We say it is extremal with respect to the direction~$(1,1)$ 
at the point~$x$ provided~$x$ is an endpoint of a segment~$\ell$ such that its second endpoint lies on the boundary 
of~$\Sigma$\textup, \;$\ell$ is parallel to~$(1,1)$\textup, the function~$G|_\ell$ is affine\textup, and~$G$ is differentiable 
in the direction~$(1,1)$ at~$x$.

The definition of a function extremal with respect to the direction~$(1,-1)$ is completely similar.
\end{Def}
\begin{Th}[Simplification of Corollary~$1.25$ in~\cite{Novikov2022}]\label{MishasTh}
Assume~$G\colon \Sigma\to \R$ is an upper semicontinuous diagonally concave function whose discontinuity set is discrete. {Let it also fulfill the bound~$|G(x)| \lesssim e^{|x_1|/\eps}$ for some~$\eps > 1$.} Assume for any~$x$ in the interior of~$\Sigma$ one of the following possibilities occur\textup:
\begin{enumerate}[1)]
\item the function~$G$ is bilinear in a neighborhood of~$x$\textup;
\item the function~$G$ is affine {with respect to~$(1,1)$} and extremal with respect to~$(1,1)$ at~$x$\textup;
\item the function~$G$ is affine {with respect to~$(1,-1)$} and extremal with respect to~$(1,-1)$ at~$x$\textup;
\item the function~$G$ is extremal with respect to both~$(1,1)$ and~$(1,-1)$ at~$x$.
\end{enumerate}
Then\textup,~$G$ is the minimal diagonally concave function \textup(with respect to its boundary values\textup).
\end{Th}
Usually, the construction of the Bellman function goes as follows. First, we guess some function~$B$ that has the required concavity properties (say, it is locally concave or diagonally concave) and is also somehow flat (solves~\eqref{MongeAmpere} or~\eqref{StrangeEquation}).  Usually,~$B$ is called a \emph{Bellman candidate}. To make a guess, we study the local structure of possible Bellman candidates and then try to construct a global Bellman candidate from the local ones. From concavity, we get that~$B$ is not smaller than the Bellman function. To obtain the reverse inequality, we need to construct an \emph{optimizer}~$\phi_x$ for any~$x$ in the domain of our Bellman function. The optimizers are the functions that the functional we optimize in the definition of the Bellman function attains its maximal value at. Say, for the optimization problem~\eqref{BellmanFunctionInTheStrip}, we wish to find a function~$\zeta_x$ such that
\eq{
\av{f(\zeta_x)}{[0,1]} =B(x), \quad \Big(\av{\zeta_x}{[0,1]},\av{\zeta_x^2}{[0,1]}\Big) = x,\quad \text{and}\quad \zeta_x \in \BMO_\eps.
}
If such a function is constructed, then we obtain the reverse inequality~$\B(x)\geq B(x)$ from the definition of~$\B$ and prove that~$\B = B$. This gives us the desired analytic expression for~$\B$ since~$B$ is defined by an explicit formula.

The problem is that sometimes the optimizers do not exist, i.\,e., the supremum in the formula for the Bellman function is not attained as a value at certain function. This is quite often the case for the optimization problems~\eqref{VDef}. The construction of the corresponding optimizers (especially for the cases of foliations considered in Section~\ref{S4}) might be quite laborious. Theorem~\ref{MishasTh} allows to avoid working with optimizers. 

The verification of concavity of a Bellman candidate is usually a routine problem. The lemma below often simplifies the computations.
\begin{Le}\label{affineLemma}
Let~$G$ be a differentiable function on a subdomain of~$\R^2$ that is affine with respect to one of the directions~$(1,-1)$ or~$(1,1)$. Let~$H$ be its directional derivative with respect to the other direction. Then\textup,~$H$ is also affine along the same direction as~$G$.  
\end{Le}  
\begin{proof}
Without loss of generality,~$G(x) = (x_1 - x_2)k(x_1+x_2) + b(x_1 + x_2)$, where~$k$ and~$b$ are functions of a single variable. Then,~$H(x) = 2(x_1 - x_2)k'(x_1+x_2) + 2b'(x_1 + x_2)$ is an affine function on the lines~$x_1 + x_2 = const$.
\end{proof}

\subsection{Plan of the paper}\label{s16}
In the forthcoming Section~\ref{S2}, we prove Theorem~\ref{TheoremForSubordination} using soft methods from~\cite{StolyarovZatitskiy2016}. As a by-product of these reasonings, we also obtain the inequality~$\VB(x_1,0) \leq \B(x_1,x_1^2 + 1)$ in Theorem~\ref{TheoremForDiagonal}, see Corollary~\ref{InequalityOnDiagonal} below. Most of the remaining part of the paper contains the proof of the reverse inequality~$\VB(x_1,0) \geq \B(x_1,x_1^2 + 1)$ and the computation of the function~$\VB$. The exposition is split into several parts, each part is devoted to a specific figure in the foliation for~$\B$. We recall the formulas for the solution of the Monge--Amp\`ere equation generated by the corresponding foliation and also the conditions for the local concavity of the constructed solution. After that we describe the corresponding foliation in~$\Sigma$ and show that exactly the same conditions provide diagonal concavity of the solution to~\eqref{StrangeEquation}. Note that all these reasoning work for the case where~$f$ fulfills Conditions~\ref{reg} and~\ref{sum} only. 

In Section~\ref{S3}, we describe tangent domains (see Subsections~\ref{s31} and~\ref{s32}) and the simplest linearity domain, the angle (see Section~\ref{s33}). The tangents in~$\Omega$ correspond to {foliations called }horizontal herringbones in~$\Sigma$ and angles correspond to bilinearity domains in~$\Sigma$ called squares. We conclude Section~\ref{S3} with examples. 

Section~\ref{S4} contains the description of chordal domains and their relatives in~$\Sigma$ called vertical herringbones. In Subsection~\ref{s41} we recall the structure of a chordal domain and describe the vertical herringbones. In Subsection~\ref{s42}, we show that the conditions for a~$C^1$-smooth concatenation of a function constructed from a chordal domain and tangent domains are the same as the conditions for a~$C^1$-smooth concatenation of a function constructed from horizontal and vertical herringbones. We finish Section~\ref{S4} with an example.

In Section~\ref{S5}, we describe the other linearity domains in~$\Omega$ and the corresponding bilinearity domains in~$\Sigma$. Section~\ref{s51} contains the theory for linearity domains with two points on the lower boundary called trolleybuses and their relatives, called corners, in~$\Sigma$. Section~\ref{s52} is devoted to multifigures. Those are linearity domains that have more than two points on the lower boundary. Here the presentation is more condensed, because the reasonings are usually similar to previous cases.

Section~\ref{S6} concludes the proof of Theorem~\ref{TheoremForDiagonal}. In Subsection~\ref{s61}, we provide a digest of the combinatorial description of a foliation in~$\Omega$. Since we have proved~\eqref{TheoremForDiagonalFormula} for any individual figure, the validity of this formula for the whole Bellman function simply follows from the fact that the Bellman function is defined via described foliation. This was proved in~\cite{ISVZ2018}. Note that to invoke the theory from~\cite{ISVZ2018}, we need to assume~$f$ satisfies Conditions~\ref{reg} and~\ref{sum}. Thus, we first prove Theorem~\ref{TheoremForDiagonal} for such~$f$ and then, in Section~\ref{s62}, reduce the general case to the considered one by means of an approximation argument. 

We provide auxiliary information in Section~\ref{S7}. It is split into three subsections. The first one is about automatic concavity. It appears that some of our minimal functions are concave in additional directions. The second one is about smoothness. Since our Bellman functions solve the homogeneous Monge--Amp\`ere equation (or more involved equations), it is natural to ask how do their regularity properties depend on the regularity properties of the cost functions. We do not provide any results in this direction, only indicate the problems and related research. The third topic is about numerical computations for some of our functions. {We finish the paper with a short list of conjectures and open questions.}

\section{$\Omega$-martingales}\label{S2}
\subsection{Martingale representation of minimal locally concave functions}\label{s21}
We recall two useful notions from~\cite{StolyarovZatitskiy2016}.
\begin{Def}
Let~$\omega$ be a subset of~$\R^d$. We call a point~$y \in \partial \omega$ locally extremal if there are no open segments~$\ell \subset \cl \omega$ such that~$y \in \ell$. The set of all locally extremal points is called the fixed boundary of~$\omega$ and denoted by~$\FixedBoundary \omega$. The set~$\FreeBoundary \omega = \partial \omega \setminus \FixedBoundary \omega$ is called the free boundary.
\end{Def}
In the case of the parabolic strip~$\Omega_\eps$ defined in~\eqref{ParabolicStrip}, the set~$\set{y\in \R^2}{y_2 = y_1^2}$ is the fixed boundary and~$\set{y\in \R^2}{y_2 = y_1^2 + \eps^2}$ is the free boundary. 
\begin{Def}\label{BasicMartingales}
Let~$\omega \subset \R^d$ be a closed set. An~$\R^d$-valued martingale~$M$ adapted to some filtration~$\{\F_n\}_n$ is called an~$\omega$-martingale if it satisfies the conditions listed below.
\emph{\begin{enumerate}
\item \emph{There exists a random variable~$M_{\infty}$ with values in~$\FixedBoundary \omega$ such that
\eq{\label{LevyMart}
\E |M_{\infty}| < \infty \quad \hbox{and}\quad M_n = \E(M_{\infty}\mid \F_n).
}
}
\item \emph{For every~$n \in \{0\}\cup \N$ and every atom~$a \in \F_n$ there exists a convex set~$C_a \subset \omega$ such that~$M_{n+1}|_a$ lies inside~$C_a$ almost surely.}
\end{enumerate}
}
\end{Def} 
An~$\omega$-martingale is a martingale that wanders inside~$\omega$ and stops at~$\FixedBoundary\omega$. We also prevent it from `jumping over the boundary', this is done with the help of the sets~$C_a$. The result below will shortly imply Theorem~\ref{TheoremForSubordination} and the inequality~$\VB(x_1,0) \leq \B(x_1,x_1^2+1)$ in Theorem~\ref{TheoremForDiagonal}. 
\begin{Th}\label{OmegaMartTh}
Let~$\{\psi_n\}_n$ be a martingale transform of a martingale~$\{\varphi_n\}_n$ uniformly bounded by~$1$. Then\textup, 
the~$\R^2$\!-valued martingale
\eq{\label{OurMartingale}
M_n = \Big(\psi_n, \E(\psi_\infty^2\mid \F_n)\Big),\quad n \geq 0,
}
is an~$\Omega$-martingale.
\end{Th}
The proof will be based upon the following geometric lemma.
\begin{Le}\label{ShrinkingLemma}
Assume that the convex hull of the points~$y_k = (y_{k,1},y_{k,2})\in \R^2$\textup, $k=1,2,\ldots, m$\textup, lies in~$\Omega_1$. Let the points~$z_k = (z_{k,1},z_{k,2})\in \Omega_1$,~$k=1,2,\ldots,m$\textup, be such that
\alg{
\forall j,k \quad\  |z_{k,1}-z_{j,1}|\leq |y_{k,1}-y_{j,1}|\,;\\
\forall k \qquad z_{k,2} - z_{k,1}^2 \leq y_{k,2} - y_{k,1}^2.\ \;
}
Then\textup, the convex hull of the points~$z_k$ also lies inside~$\Omega$.
\end{Le}
\begin{proof}
Let~$\{\alpha_k\}_{k=1}^m$ be positive numbers with sum one, let~$\bz = \sum_{k=1}^m\alpha_k z_k$. We wish to prove that~$\bz_2-\bz_1^2 \leq 1$.
Let us estimate this expression:
\mlt{
\bz_2-\bz_1^2 = \sum\limits_{k=1}^m\alpha_k z_{k,2} - \Big(\sum\limits_{k=1}^m \alpha_kz_{k,1}\Big)^2\! 
= \sum\limits_{k=1}^m\alpha_k(z_{k,2} - z_{k,1}^2) + \sum\limits_{k=1}^m(\alpha_k - \alpha_k^2)z_{k,1}^2 
- \!\!\!\sum\limits_{k,j\colon k\ne j}\!\!\!\alpha_k\alpha_jz_{k,1}z_{j,1} \\
=\sum\limits_{k=1}^m\alpha_k(z_{k,2}-z_{k,1}^2) + \tfrac12\!\!\!\sum\limits_{k,j\colon k\ne j}\!\!\!\alpha_k\alpha_j(z_{k,1}-z_{j,1})^2,
}
since~$\alpha_k - \alpha_k^2 = \alpha_k(1-\alpha_k)\, = \!\!\!\sum\limits_{j\colon j\ne k}\!\!\!\alpha_k\alpha_j$. By a similar computation,
\eq{
\by_2-\by_1^2 \,= \sum\limits_{k=1}^m\alpha_k(y_{k,2}-y_{k,1}^2) 
\,+\,\tfrac12 \!\!\!\sum\limits_{k,j\colon k\ne j}\!\!\!\alpha_k\alpha_j(y_{k,1}-y_{j,1})^2;
}
here~$\by = (\by_1,\by_2)=\sum_{k=1}^m\alpha_k y_{k}\in\Omega_1$. Consequently,~$\bz_2-\bz_1^2 \leq \by_2-\by_1^2 \leq 1$, and the lemma is proved.
\end{proof}
\begin{proof}[Proof of Theorem~\emph{\ref{OmegaMartTh}}.]
{The definition of an~$\omega$-martingale consists of two parts. The verification of the first part (that~$M$ given by~\eqref{OurMartingale} is generated by its limit random variable attaining values in the parabola~$y_2 = y_1^2$, which is the fixed boundary of~$\Omega$) is simpler. We set~$M_{\infty} = (\psi_{\infty},\psi_{\infty}^2)$
and rely upon Doob's theorem to verify~\eqref{LevyMart}. To justify the second part of the definition,
} we will be using the formula
\eq{\label{MartDif}
\E\big((\phi_\infty - \phi_n)^2\mid \F_n\big) = \E\Big(\sum\limits_{k > n}(d\phi_k)^2\mid \F_n\Big),\quad n \geq 0,
}
here~$\phi = \{\phi_n\}_n$ is an arbitrary~$L_2$-martingale with the limit value~$\phi_\infty$. It yields the inequality
\eq{\label{L2DiffSubord}
\E((\psi_\infty - \psi_n)^2\mid \F_n) \leq \E((\varphi_\infty - \varphi_n)^2\mid \F_n)
}
for all~$n \geq 0$. Let now~$a\in \F_n$ be an atom, let~$a_1,a_2,\ldots,a_m$ be all its kids, i.\,e., the atoms of~$\F_{n+1}$ lying inside~$a$. Denote
\eq{
\begin{aligned}
&y_{k,1} = \varphi_{n+1}(a_k),\quad  y_{k,2} = \E(\varphi_\infty^2\mid \F_{n+1})(a_k),\\
&z_{k,1} = \psi_{n+1}(a_k),\quad  z_{k,2} = \E(\psi_\infty^2\mid \F_{n+1})(a_k),
\end{aligned} \qquad k=1,2,\ldots,m;
}
{here we allow a slight abuse of notation treating an~$\F_n$-measurable random variable as a function defined on the atoms of~$\F_n$.}
Since
\eq{\label{VarianceFormula}
\E((\phi_\infty - \phi_n)^2\mid \F_n) = \E(\phi_\infty^2\mid \F_n) - \phi_n^2
}
for any martingale~$\phi$,  we have, in view of~\eqref{L2DiffSubord},
\eq{
0 \leq z_{k,2} - z_{k,1}^2 \leq y_{k,2} - y_{k,1}^2 \leq 1.
}
The latter inequality in the chain is a consequence of the requirement~$\|\varphi_\infty\|_{L_\infty}\leq 1$. We also note that for any~$k,j = 1,2,\ldots, m$,
\eq{\label{MartingaleTransformMultiplFormula}
|z_{k,1} - z_{j,1}| = |\alpha(a)||y_{k,1} - y_{j,1}| \leq |y_{k,1} - y_{j,1}|.
}
Since~$|\varphi_{\infty}| \leq 1$ almost surely, the points~$y_k$ lie inside the set
\eq{
\Set{{\bf y} = ({\bf y_1, y_2})\in \R^2}{{\bf y_1^2} \leq \bf{y_2} \leq 1}.
} 
Therefore, their convex hull belongs to~$\Omega$. By Lemma~\ref{ShrinkingLemma}, the convex hull of the points~$z_k$ lies inside~$\Omega$ as well, which proves that~$M_n$ is an~$\Omega$-martingale.
\end{proof}

\subsection{Proof of Theorem~\ref{TheoremForSubordination}}\label{s22}
The main theorem of~\cite{StolyarovZatitskiy2016} says that
\eq{\label{OmegaMartFormula}
\B(y) = \sup\Set{\E f(M_{\infty,1})}{M_0 = y,\ M = \{(M_{n,1}, M_{n,2})\}_n\ \text{is an $\Omega$-martingale}},\quad y\in \Omega,
}
provided~$f$ is locally bounded and globally bounded from below; {here~$M_{\infty,1}$ is the scalar random variable being the first coordinate of the~$\R^2$-valued random variable~$M_\infty$.} We will rely upon this formula in the forthcoming  proof.
\begin{proof}[Proof of Theorem~\emph{\ref{TheoremForSubordination}}]
Let~$\varphi$ and~$\psi$ be the same as in~\eqref{UDef}. By Theorem~\ref{OmegaMartTh},~$M_n$ {given by~\eqref{OurMartingale}}
is an~$\Omega$-martingale. Therefore, 
\eq{\label{eq222}
\E f(\psi_\infty) = \E f(M_{\infty,1}) \Leqref{OmegaMartFormula} \B(M_0)= \B(\psi_0, \E \psi_\infty^2). 
}
Since~$\psi_0 = x_1$,~$\varphi_0 = x_2$, and
\eq{\label{L2Estimate}
\E\psi_\infty^2 \Eeqref{VarianceFormula} \E(\psi_\infty - \psi_0)^2 + \psi_0^2 \Leqref{L2DiffSubord} \E(\varphi_\infty - \varphi_0)^2 + x_1^2 \Eeqref{VarianceFormula} \E\varphi_\infty^2 - x_2^2 + x_1^2,
} 
the value~$\E\psi_\infty^2$ lies between~$x_1^2$ and~$x_1^2+1-x_2^2$ and we obtain from~\eqref{eq222}
\eq{
\E f(\psi_\infty) \leq \sup\limits_{0\leq \delta \leq 1-x_2^2} \B(x_1,x_1^2+\delta).
}
A passage to the supremum on the left hand side finishes the proof.
\end{proof}
\begin{Que}\label{Que1}
When does~\eqref{FirstMajorization} turn into equality\textup?
\end{Que}

The corollary below provides `a half' of Theorem~\ref{TheoremForDiagonal}. This `half' is simpler than the other one.
\begin{Cor}\label{InequalityOnDiagonal}
For any~$x\in \Sigma$\textup, we have
\eq{\label{SecondMajorization}
\VB(x) \leq  \B(x_1,x_1^2+1-x_2^2).
}
\end{Cor}
\begin{proof}
The proof is completely similar to the reasonings above, the only difference is that the inequality~\eqref{L2Estimate} turns into equality:
\eq{
\E\psi_\infty^2 = \E(\psi_\infty - \psi_0)^2 + \psi_0^2 \Eref{\scriptscriptstyle\eqref{MartDif}} \E(\varphi_\infty - \varphi_0)^2 + x_1^2 = \E\varphi_\infty^2 - x_2^2 + x_1^2 = 1-x_2^2+ x_1^2.
}
\end{proof}

\section{Simplest case: tangent domains}\label{S3}
\subsection{Tangent domain}\label{s31}
We study the simplest foliation for the function~$\B$ that consists of tangents. See~\cite{ISVZ2018} for more details. Any tangent line to the free boundary of~$\Omega$ is split by the tangency point into two `halves': the right (that lies on the right of the tangency point) and the left. A part of the tangent line that is right and lies inside~$\Omega$ is called the right tangent; similar for the left case. The right tangents foliate~$\Omega$. We consider the functions defined on a part of~$\Omega$ for which the tangents are the extremal segments. For any~$y\in\Omega$, consider the right tangent passing through~$y$. Let it intersect the fixed boundary at~$(\ur,\ur^2)$,~${\ur = \ur(y)}$. Then, its first coordinate 
may be expressed as
\eq{\label{urFormula}
\ur(y) = y_1 + 1 - \sqrt{y_1^2 + 1 - y_2},\qquad y=(y_1,y_2)\in\Omega.
}
The fact that a function~$B$ is affine along the right tangents and fulfills the boundary condition~\eqref{BCForBMO} is expressed by the formula
\eq{\label{linearity}
B(y) = m(u)(y_1 - u) + f(u), 
}
where~$u=\ur(y)$ and~$m = \Mrt\colon \R\to \R$ is an arbitrary function. We cite Proposition~$3.2.1$ in~\cite{ISVZ2018} with a slight change of notation; in the proposition below,~$u_1$ and~$u_2$ may be infinite. See also Fig.~\ref{TangHorRight} below for the visualization of the symmetric case of left tangents.
\begin{St}[Proposition~$3.2.1$ in~\cite{ISVZ2018}]\label{321ISVZ}
Suppose~$m$ satisfies the differential equation
\begin{equation}\label{difeq}
m'(u) + m(u) - f'(u) = 0
\end{equation}
and the inequality~$m''(u) \leq 0$ for~$u \in [u_1,u_2]$. Then the function~$B$ given by~\eqref{linearity} is a locally concave solution to the homogeneous Monge--Amp\`ere equation~\eqref{MongeAmpere} on the domain
\begin{equation*}
\Rt(u_1,u_2) = \{y \in \Omega \mid \ur(y) \in [u_1,u_2]\}.
\end{equation*}
\end{St}
It is important that the function~$B$ is defined by~\eqref{linearity} and~\eqref{difeq} only up to a constant function: we are free to choose the initial values for the differential equation~\eqref{difeq}. Usually, we will be specifying~$m(u_1)$. 

\begin{St}\label{HorizontalFurRight}
Assume~$m''(v) \leq 0$ for any~$v\in [v_1,v_2]$ and the function~$m$ satisfies~\eqref{difeq}. 
Then\textup, the function~$V$ given by
\eq{\label{BFRT}
V(x_1,x_2) = m(v)(x_1-v) + f(v) ,\quad \text{where}\quad v(x)=\vr(x) = x_1 + 1 - |x_2|,
}
is a diagonally concave solution to the equation~\eqref{StrangeEquation} on the domain
\eq{
\SRt(v_1,v_2) = \set{x\in \Sigma}{\vr(x)\in [v_1,v_2]}.
}
\end{St}

We see that the function~$V$ is affine on segments of direction~$(1,1)$ when~$x_2 > 0$ and affine on segments of direction~$(1,-1)$ when~$x_2 < 0$. We will call this foliation in~$\Sigma$ the \emph{right horizontal herringbone}. The left horizontal herringbone is symmetric. It is described in Proposition~\ref{HorizontalFurLift} below; see Fig.~\ref{TangHorRight} for visualization.

\begin{proof}
We write~$v=\vr$ for brevity, compute the derivative of~$V$ w.r.t.~$x_2$,
\eq{\label{Vx2}
\frac{\partial V}{\partial x_2} = \Bigg\{\begin{aligned}
-f'(v) - m'(v)(x_2-1) + m(v),\quad &x_2 \geq 0\\
f'(v) - m'(v)(x_2+1) - m(v),\quad &x_2 < 0
\end{aligned}\Bigg\}
\ \Eeqref{difeq} -m'(v)x_2
}
and realize that~$V\in C^1$. By symmetry and~$C^1$-smoothness, it suffices to verify the diagonal concavity on the part of~$\SRt(v_1,v_2)$ where~$x_2 < 0$. On that domain
\alg{
\label{Vx1}\frac{\partial V}{\partial x_1} &= f'(v) - (x_2+1)m'(v) \Eeqref{difeq} m(v) - m'(v) x_2;\\
\frac{\partial^2 V}{\partial x_1^2} &= m'(v) - m''(v) x_2;\\
\frac{\partial^2 V}{\partial x_1\partial x_2} &= - m''(v)x_2;\\
\frac{\partial^2 V}{\partial x_2^2} &= -m'(v) - m''(v)x_2.
}
Therefore, the second derivative of~$V$ in the direction~$(\beta,1)$ is equal to
\eq{\label{SecondDerivatieveBeta}
-(1+\beta)^2 x_2 m''(v) + (\beta^2-1)m'(v).
}
We see that the concavity in the direction~$(1,1)$ holds exactly when~$m''(v) \leq 0$ for every~$v\in [v_1,v_2]$ (recall~$x_2 < 0$). This means~$V$ is diagonally concave.
\end{proof}

We state a symmetric proposition for the left tangents. See Fig.~\ref{TangHorRight} for visualization of Proposition~\ref{HorizontalFurLift}. Now let
\eq{\label{ulFormula}
\ul(y) = y_1 - 1 + \sqrt{y_1^2 + 1 - y_2},\qquad y=(y_1,y_2)\in\Omega.
}

\begin{St}[Proposition~$3.2.2$ in~\cite{ISVZ2018}]\label{322ISVZ}
Suppose that~$m$ satisfies the differential equation
\begin{equation}\label{difeq2}
-m'(u) + m(u) - f'(u) = 0
\end{equation}
and the inequality~$m''(u) \geq 0$ for~$u \in [u_1,u_2]$. Then\textup,~$B$ given by~\eqref{linearity} with~$u=\ul$ is 
a locally concave solution to the homogeneous Monge--Amp\`ere equation~\eqref{MongeAmpere} on the domain
\begin{equation*}
\Lt(u_1,u_2) = \{y \in \Omega\mid \ul(y) \in [u_1,u_2]\}.
\end{equation*}
\end{St}

\begin{St}\label{HorizontalFurLift}
Assume~$m''(v) \geq 0$ for any~$v\in [v_1,v_2]$ and the function~$m$ satisfies~\eqref{difeq2}. Then\textup, the function~$V$ given by
\eq{\label{BFLT}
V(x_1,x_2) = m(v)(x_1-v) + f(v) ,\quad \text{where}\quad v(x)=\vl(x) = x_1 -1 + |x_2|,
}
is a diagonally concave solution to the equation~\eqref{StrangeEquation} on the domain
\eq{
\SLt(v_1,v_2) = \set{x\in \Sigma}{\vl(x)\in [v_1,v_2]}.
}
\end{St}

\begin{figure}[h!]
\includegraphics[height=4.5cm]{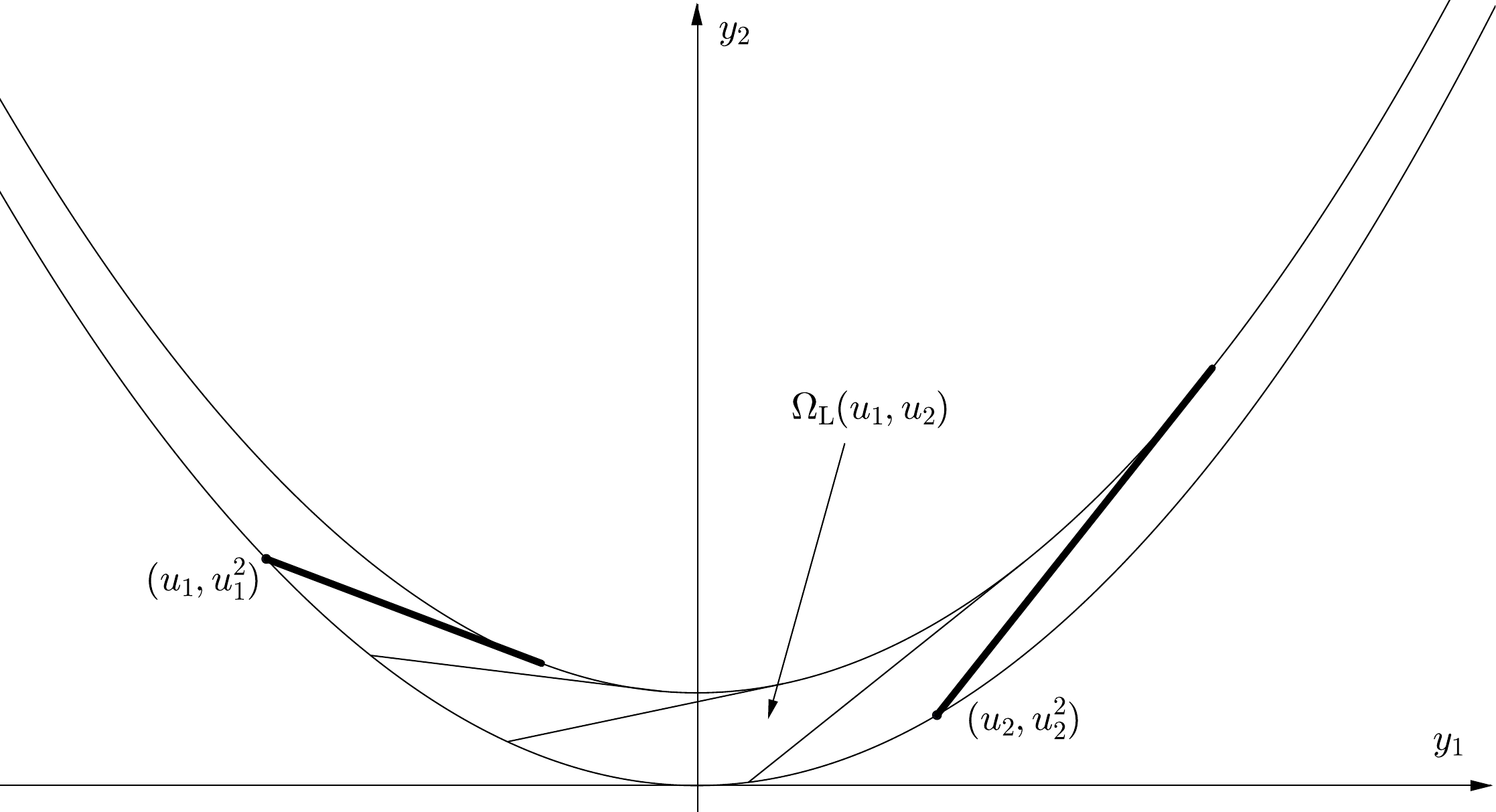}
\includegraphics[height=4.5cm]{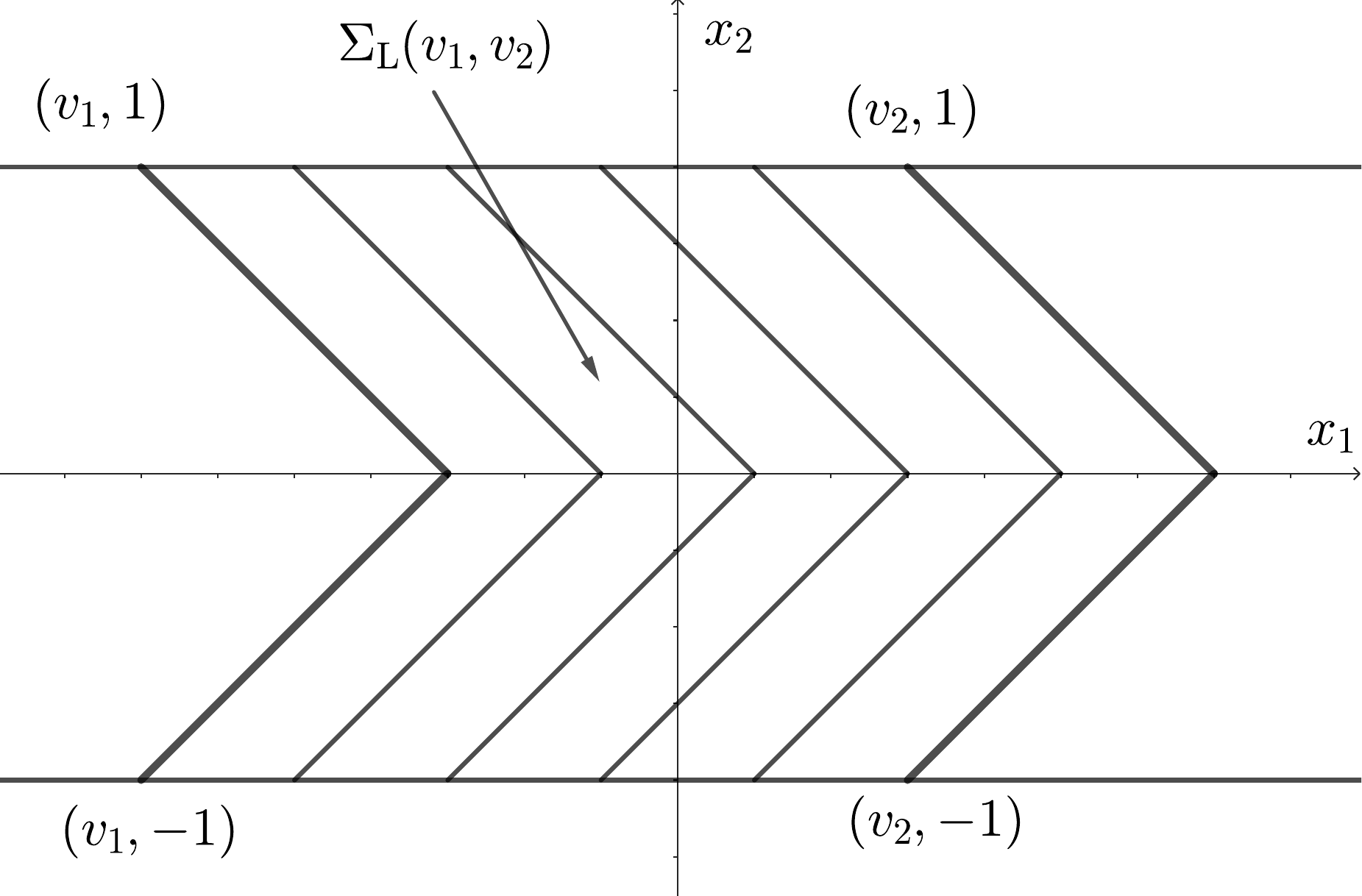}
\caption{A tangent domain and the corresponding foliation for diagonal concavity}
\label{TangHorRight}
\end{figure}

The differential equations~\eqref{difeq} and~\eqref{difeq2} have one-parameter families of solutions~$m$ for a given~$f$. Namely, a solution of~\eqref{difeq} is given by
\begin{equation}\label{ExplicitFormulaForm}
m(u) = e^{-u}\bigg(e^{u_1}m(u_1) + \int\limits_{u_1}^u f'(t)e^{t}\,dt\bigg),\qquad u \in [u_1,u_2].
\end{equation}
The inequality~$m''(u) \leq 0$ can be rewritten as
\begin{equation}\label{mr''_firstformula}
m''(u) = e^{(u_1-u)}m''(u_1) + 
e^{-u}\int\limits_{u_1}^u e^{t}\,df''(t) \leq 0.
\end{equation}

Similarly, a solution of~\eqref{difeq2} is given by
\begin{equation}\label{ExplicitFormulaForm2}
  	m(u) = e^{u}\bigg(e^{-u_2}m(u_2) + \int\limits_{u}^{u_2} f'(t)e^{-t}\,dt\bigg),\qquad u\in [u_1,u_2].
  \end{equation}
The inequality~$m''(u) \geq 0$ in this situation is
\begin{equation*}
  	m''(u) = e^{(u-u_2)}m''(u_2) + e^{u}\int\limits_{u}^{u_2} e^{-t}\,df''(t) \geq 0.
  \end{equation*}

\subsection{Infinite tangent domains}\label{s32}
We consider the case~$u_1 = -\infty$ for the right tangents separately. 

\begin{St}[Proposition~$3.2.5$ in~\cite{ISVZ2018}]\label{RightTangentsCandidateInfty}
Suppose that~$m$ is given by the formula
\begin{equation}\label{minfty}
m(u) = e^{-u}\int\limits_{-\infty}^u f'(t)e^{t}\,dt,
\end{equation}
and the inequality~$m''(u) \leq 0$ is fulfilled for~$u \in (-\infty,u_2)$. Then\textup, the function~$B$ given by~\eqref{linearity} with~$u=\ur$ is a locally concave solution to~\eqref{MongeAmpere} on~$\Rt(-\infty,u_2)$.
\end{St}

The Conditions~\ref{reg} and~\ref{sum} guarantee that the integral~\eqref{minfty} is absolutely convergent and~$m$ defined in this way satisfies~\eqref{difeq}, see Lemma~$2.1.13$ in~\cite{ISVZ2018}. This also reduces the proposition below to Proposition~\ref{HorizontalFurRight}.

\begin{St}\label{RightFurInfty}
Suppose that~$m$ given by~\eqref{minfty} fulfills the inequality~$m''(v) \leq 0$ for~$v\in (-\infty,v_2)$. Then\textup, the function~$V$ given by~\eqref{BFRT} is a diagonally concave solution to~\eqref{StrangeEquation} on~$\SRt(-\infty,v_2)$.
\end{St}

We state symmetric proposition for left tangents.

\begin{St}[Proposition~$3.2.6$ in~\cite{ISVZ2018}]\label{LeftTangentsCandidateInfty}
Suppose that~$m$ is given by the formula
\begin{equation}\label{minfty2}
m(u) = e^{u}\int\limits_u^{+\infty} f'(t)e^{-t}\,dt,
\end{equation}
and the inequality~$m''(u) \geq 0$ is fulfilled for~$u \in (u_1,+\infty)$. Then the function~$B$ given by formula~\eqref{linearity} with~$u=\ul$ is a locally concave solution to~\eqref{MongeAmpere} on~$\Lt(u_1,+\infty)$.
\end{St}

\begin{St}\label{LeftFurInfty}
Suppose that the function~$m$ given by~\eqref{minfty2} fulfills the inequality~$m''(v) \geq 0$ for \hbox{$v\in (v_1,\infty)$.} 
Then\textup, the function~$V$ given by~\eqref{BFLT} is a diagonally concave solution to~\eqref{StrangeEquation} on~$\SLt(v_1,\infty)$.
\end{St}

It will be important for us that the functions~$V$ constructed in {Propositions~\ref{RightFurInfty} and~\ref{LeftFurInfty}} satisfy the inequality
\eq{\label{BoundForMishasTheorem}
|V(x)| {\ \lesssim\ } e^{|x_1|/{\eps}}, \qquad x\in \Sigma,
}
for some~$\eps > 1$. {Indeed, since~$f$ fulfills Condition~\ref{sum}, Proposition~$2.1.12$ in~\cite{ISVZ2018} yields~$|m(u)| \lesssim e^{|u|/\eps}$ for some~$\eps > 1$, which, in its turn, implies the desired bound~\eqref{BoundForMishasTheorem} by either~\eqref{BFRT} or~\eqref{BFLT}.} Therefore, we are in position to apply Theorem~\ref{MishasTh}.

\begin{Th}\label{minftyTh}
Suppose that~$m$ given by~\eqref{minfty} fulfills the inequality~$m''(v) \leq 0$ for~$v\in \R$. Then\textup, $\VB = V$\textup, where the latter function is given by~\eqref{BFRT}.
\end{Th}

\begin{proof}
Our function~$V$ falls under the scope of Theorem~\ref{MishasTh} by construction: a point~$x\in \Sigma$ with~$x_2 > 0$ fulfills the second assumption of that theorem, a point with~$x_2 < 0$ fulfills the third assumption, and the points on the midline~$x_2=0$ fulfill the fourth. Thus, $\VB(x) = V(x)$ for all~$x\in \Sigma$.
\end{proof}

\begin{Th}\label{minftyTh2}
Suppose that~$m$ given by~\eqref{minfty2} fulfills the inequality~$m''(v) \geq 0$ for~$v\in \R$. Then\textup, $\VB = V$\textup, where the latter function is given by~\eqref{BFLT}.
\end{Th}
In the case of Theorem~\ref{minftyTh} the inequality~$m''(v) \leq 0$ turns into
\begin{equation*}
m''(v) =  e^{-v}\int\limits_{-\infty}^v e^{t}\,df''(t) \leq 0.
\end{equation*}
In the case of Theorem~\ref{minftyTh2} the inequality~$m''(v) \geq 0$ turns into
\begin{equation*}
m''(v) = e^{v}\int\limits_{v}^{+\infty} e^{-t}\,df''(t) \geq 0.
\end{equation*}
\begin{Rem}\label{RemarkTangents}
As we see, the statement of Theorem~\textup{\ref{TheoremForDiagonal}} holds true in the case where one of 
Theorems~\textup{\ref{minftyTh}} and~\textup{\ref{minftyTh2}} is applicable. Consider\textup, for example\textup, 
the case of right tangents and right horizontal herringbone. Theorem~$3.9$ in~\textup{\cite{IOSVZ2016}} says that if 
the assumptions of Theorem~\textup{\ref{RightTangentsCandidateInfty}} are satisfies with~$u_2=\infty$\textup, 
then the function~$B$ constructed in this theorem coincides with~$\B$. Therefore\textup,
\eq{
\VB(x_1,0) \Eeqref{BFRT} m(x_1+1) + f(x_1 + 1) \Eref{\scriptscriptstyle\eqref{urFormula},\,\eqref{linearity}} \B(x_1,x_1^2 + 1),
} 
which proves Theorem~\textup{\ref{TheoremForDiagonal}} in this case.
\end{Rem}

\subsection{Angle}\label{s33}
Now we wish to investigate the simplest linearity domain, the angle. This is a domain~$\Ang(w)$ in~$\Omega$ that is bounded by a right tangent from the left, by a left tangent from the right, and has only one point~$W=(w,w^2)$ on the fixed boundary, see {the left drawing on} Fig.~\ref{AngleEnvelope} below. The point~$W$ is the common endpoint for the two tangents bounding~$\Ang(w)$. We assume that the function~$B$ is affine in~$\Ang(w)$:
\eq{\label{InAngle}
B(y) = \beta_2y_2 + \beta_1y_1 + \beta_0,\qquad y\in \Ang(w),
}
and is defined by the formulas
\eq{\label{MrtAndMlt}
\begin{aligned}
B(y) &= \Mrt(\ur)(y_1 - \ur) + f(\ur),&\qquad &\ur = \ur(y),&\qquad &y \in \Rt(u_1,w);\\
B(y) &= \Mlt(\ul)(y_1 - \ul) + f(\ur),& \qquad &\ul = \ul(y),&\qquad &y \in \Lt(w,u_2),
\end{aligned}
}
in the adjacent tangent domains.
We also assume that the functions~$\Mrt$ and~$\Mlt$ satisfy the assumptions of Propositions~\ref{321ISVZ} and~\ref{322ISVZ}, respectively, and ask the question when the function~$B$ is~$C^1$-smooth and locally concave. Proposition~$3.4.4$ in~\cite{ISVZ2018} gives the answer.

\begin{St}[Proposition~$3.4.4$ in~\cite{ISVZ2018}]\label{344ISVZ}
Assume~$w\in\R$ solves the balance equation
\eq{\label{BalanceEquation}
\Mrt(w) + \Mlt(w) = 2f'(w). 
}
If the functions~$\Mrt$ and~$\Mlt$ satisfy the assumptions of Propositions~\textup{\ref{321ISVZ}} and~\textup{\ref{322ISVZ},} respectively\textup, then there exist~$\beta_0$\textup, $\beta_1$\textup, and~$\beta_2$ such that~$B$ defined by~\eqref{InAngle} and~\eqref{MrtAndMlt} is a~$C^1$-smooth locally concave function {satisfying~\eqref{MongeAmpere} on each of the domains~$\Rt(u_1,w)$\textup, $\Ang(w)$\textup, and~$\Lt(w,u_2)$}.
\end{St}
The parameters in~\eqref{InAngle} are then defined by
\alg{
\label{Beta2}\beta_2 &= \frac{\Mlt(w) - \Mrt(w)}{4};\\
\label{Beta1}\beta_1 &= f'(w) - 2\beta_2w;\\
\label{Beta0}\beta_0 &= f(w) - wf'(w) + \beta_2w^2,
}
see formulas~$(3.4.2)$ and~$(3.4.5)$ in~\cite{ISVZ2018}. Now we wish to find an analog of this foliation for diagonally concave functions. As we have said in the introduction, the linearity domains in~$\Omega$ correspond to bilinearity domains in~$\Sigma$.
\begin{figure}[h!]
\includegraphics[height=4cm]{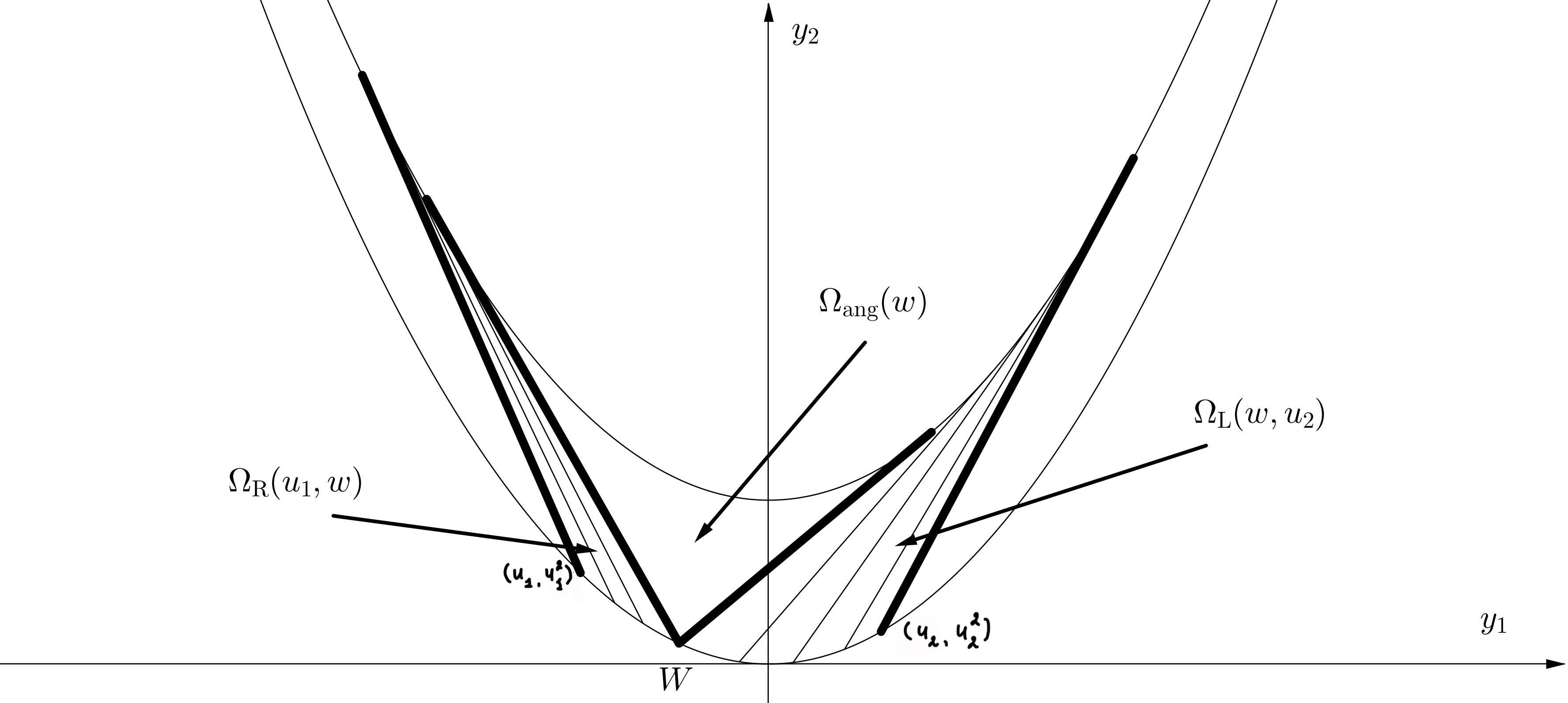}
\includegraphics[height=4.4cm]{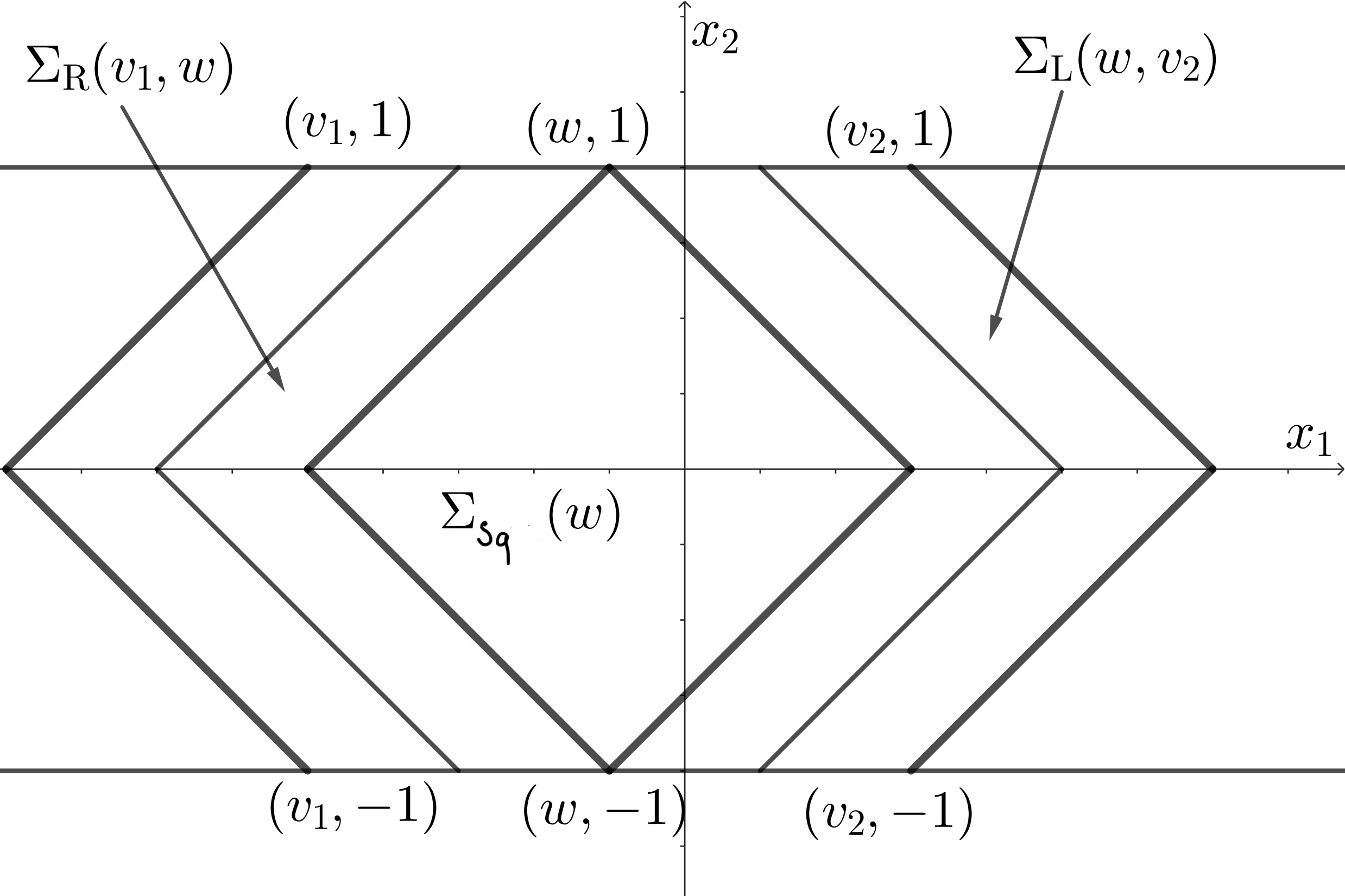}
\caption{An angle and a square}
\label{AngleEnvelope}
\end{figure}
We choose~$w\in \R$ and define the \emph{square} domain
\eq{
\Sang(w) = \Set{x\in \Sigma}{|x_1 - w| {+ |x_2|} \leq 1},
}
see Fig.~\ref{AngleEnvelope}, and assume that
\eq{\label{BilinearInEnvelope}
V(x) = \alpha_{11}(x_1^2 - x_2^2) + \alpha_1 x_1 + \alpha_0, \qquad x\in \Sang(w).
}
There is no linear term with~$x_2$ due to symmetry.

\begin{St}\label{TransferToSquare}
Assume the point~$w$ solves the balance equation~\eqref{BalanceEquation}.
If the functions~$\Mrt$ and~$\Mlt$ satisfy the assumptions of Propositions~\textup{\ref{HorizontalFurRight}} 
and~\textup{\ref{HorizontalFurLift},} respectively\textup, then there exist~$\alpha_{11}$\textup, $\alpha_1$\textup, 
and~$\alpha_0$ such that~$V$ defined by~\eqref{BFRT} on~$\SRt(v_1,w)$ with~$m = \Mrt$\textup, by~\eqref{BilinearInEnvelope} 
on~$\Sang(w)$\textup, and by~\eqref{BFLT} on~$\SLt(w,v_2)$ with~$m = \Mlt$ is a~$C^1$\!-smooth diagonally concave 
function\textup, solving~\eqref{StrangeEquation} on each of the domains~$\SRt(v_1,w)$\textup, $\Sang(w)$\textup, 
and~$\SLt(w,v_2)$.
\end{St}

\begin{proof}
We set
\eq{\label{Alphas}
\begin{aligned}
\alpha_{11} &= \beta_2;\\
\alpha_1 &= \beta_1 = f'(w) - 2\beta_2 w;\\
\alpha_0 &= \beta_0 + \beta_2 = \beta_2(1+w^2)+ f(w) - wf'(w), 
\end{aligned}
}
where~$\beta_i$ are defined in~\eqref{Beta2},~\eqref{Beta1}, and~\eqref{Beta0}. {Note that while~$\alpha_{11}$ and~$\alpha_1$ coincide with~$\beta_2$ and~$\beta_1$, respectively, the coefficient~$\alpha_0$ does not equal~$\beta_0$.}
The function~$V$ is a~$C^1$-smooth diagonally concave solution to~\eqref{StrangeEquation} on each of the domains~$\SRt(u_1,w)$,~$\Sang(w)$, and~$\SLt(w,u_2)$. To prove the theorem, it suffices to check the~$C^1$-smoothness of~$V$. 

We start with the verification of its continuity at the common boundaries of these domains. By linearity and symmetry, we may investigate the values at the points~$(w,-1)$ and $(w+1,0)$ only.

The value of~\eqref{BilinearInEnvelope} at~$(w,-1)$ is
\mlt{
\alpha_{11}(w^2-1) + \alpha_1w + \alpha_0\\ \Eeqref{Alphas}  \beta_2(w^2-1) + (f'(w) - 2\beta_2w)w + \beta_2(w^2+1) +f(w) - wf'(w) = f(w).
}
This coincides with the values of~\eqref{BFLT} and~\eqref{BFRT} at the same point.

The value of~\eqref{BilinearInEnvelope} at~$(w+1,0)$ is
\mlt{
\alpha_{11}(1+w)^2 + \alpha_1(w+1) + \alpha_0\\ \Eeqref{Alphas}  \beta_2(1+w)^2 + (f'(w) - 2\beta_2w)(w+1) + \beta_2(w^2+1) +f(w) - wf'(w) 
= 2\beta_2 + f'(w) + f(w)\\ \Eeqref{Beta2} \frac{\Mlt(w) - \Mrt(w)}{2} + f'(w) +f(w) \Eeqref{BalanceEquation} f(w) + \Mlt(w).
}
This coincides with the value of~\eqref{BFLT} at the same point.

Lemma~\ref{affineLemma} allows to check the coincidence of gradients of pieces of~$V$ at the points~$(w,-1)$ and~$(w-1,0)$ only (we have also used symmetry). Moreover, it suffices to verify the coincidence of the derivatives with respect to~$x_1$ only, since the derivatives along directions of linearity already coincide. 

The derivative w.r.t.~$x_1$ of~\eqref{BilinearInEnvelope} at~$(w,-1)$ equals~$2\alpha_{11}w + \alpha_1$, which, by~\eqref{Alphas}, equals~$f'(w)$; the desired coincidence of derivatives at~$(w,-1)$ is verified. 

The derivative w.r.t.~$x_1$ of~\eqref{BilinearInEnvelope} at~$(w-1,0)$ equals
\eq{
2\alpha_{11}(w-1) + \alpha_1 \Eeqref{Alphas} -2\beta_2 + f'(w) \stackrel{\scriptscriptstyle{\eqref{BalanceEquation},\eqref{Beta2}}}{=} \Mrt(w).
}
This coincides with the value of~\eqref{Vx1} at the same point.
\end{proof}

\begin{Rem}
In the proposition above one or both of~$v_1$ and~$v_2$ may be infinite. In this case\textup, we rely upon the corresponding proposition in Subsection~\textup{\ref{s32}}.
\end{Rem}

The theorem below is a direct consequence of Theorem~\ref{MishasTh}.

\begin{Th}\label{AngleTheorem}
Assume there exists~$w\in \R$ with the following properties\textup:
\emph{\begin{enumerate}[1)]
\item \emph{$w$ solves the balance equation~\eqref{BalanceEquation}
with~$\Mrt$ given by~\eqref{minfty} and~$\Mlt$ is given by~\eqref{minfty2}};
\item \emph{for any~$v < w$\textup, the inequality~$\Mrt''(v) \leq 0$ holds true};
\item \emph{for any~$v > w$\textup, the inequality~$\Mlt''(v) \geq 0$ holds true}.
\end{enumerate}
}
Then\textup, the function~$V$ defined by~\eqref{BFRT} on~$\SRt(v_1,w)$\textup, by~\eqref{BilinearInEnvelope} 
on~$\Sang(w)$\textup, and by~\eqref{BFLT} on~$\SLt(w,v_2)$\textup, coincides with~$\VB$.
\end{Th}
 
\begin{proof}
It suffices to verify the assumptions of Theorem~\ref{MishasTh}. For the points in the horizontal herringbones, this was already done in the proof of Theorem~\ref{minftyTh}. The points in~$\Sang(w)$ fulfill the first requirement of Theorem~\ref{MishasTh}.
\end{proof}

\begin{Rem}\label{RemarkAngle}
As we see, the statement of Theorem~\textup{\ref{TheoremForDiagonal}} holds true in the case where Theorem~\textup{\ref{AngleTheorem}} is applicable. Indeed\textup, in~$\Sang$\textup, we have
\eq{
\alpha_{11}x_1^2 + \alpha_1 x_1 + \alpha_0 = \beta_2(x_1^2 + 1) + \beta_1 x_1 + \beta_0
}
{by~\eqref{Alphas}}.
\end{Rem}

\subsection{Examples}\label{s34}
\paragraph{Examples for Theorems~\ref{minftyTh} and~\ref{minftyTh2}: one-sided exponential bounds.} Let~$\lambda \in (0,1)$, consider the function~$f(t) = e^{\lambda t}$. We see this function satisfies Conditions~\ref{reg} and~\ref{sum} and also falls under the scope of Theorem~\ref{minftyTh2}. The corresponding function~$\B$ for the~$\BMO$ case was firstly computed in~\cite{SlavinVasyunin2011} (see Example~$1$ in~\cite{IOSVZ2016} for a brief exposition):
\eq{
\B(y)=\Big(\frac{\lambda}{1-\lambda}\big(y_1-\ul(y)\big)+1\Big)\;e^{\lambda\ul(y)},\qquad\ul(y)=y_1-1+\sqrt{y_1^2-y_2+1}.
}
The exact formula for its martingale counterpart~$\VB$ is, seemingly, omitting in the literature. According to Theorem~\ref{minftyTh2},
\eq{\label{VBForExponent}
\VB(x)=\Big(\frac{\lambda}{1-\lambda}\big(x_1-\vl(x)\big)+1\Big)\;e^{\lambda\vl(x)},\qquad\vl(x) =x_1-1+|x_2|.
}
Using~\eqref{VDef}, we obtain the sharp inequality
\eq{
\E e^{\lambda(\psi_\infty - \psi_0)} \leq \Big(\frac{\lambda (1- |\varphi_0|)}{1-\lambda}+1\Big)\;e^{\lambda(|\varphi_0|-1)},\qquad |\varphi_\infty| = 1\ \text{a.\,s.},\ \psi\ \text{is a martingale transform of}\ \varphi.
}

\begin{Rem}\label{U=V}
One may see that in this case~$\UB = \VB$\textup, because the function~$\VB$ is concave in all the directions~$(\beta,1)$ with~$|\beta| \leq 1$. This may be proved using the `left' analog of~\eqref{SecondDerivatieveBeta}\textup: the second derivative of a function $V$ described in Proposition~\ref{HorizontalFurLift} with respect to the direction~$(\beta,1)$ is equal to
\eq{
-(1+\beta)^2 |x_2|m'' + (\beta^2-1)m'.
}
This quantity is non-positive for all $\beta\in[-1,1]$ and all $x_2\in [-1,1]$ if and only if $m'' \geq 0$ and $m' \geq 0$. For~$f(t) = e^{\lambda t}$ we have $m(t) = \frac{\lambda}{1-\lambda}e^{\lambda t}$ that fulfills these conditions provided $\lambda<1$.
\end{Rem}

Let us also consider the function $f(t) = -e^{\lambda t}$ with $\lambda \in (0,1)$. This function satisfies Conditions~\ref{reg} and~\ref{sum} and falls under the scope of Theorem~\ref{minftyTh} with $m(v) = -\frac{\lambda}{1+\lambda}e^{\lambda v}$. We have 
$$
\VB(x) = -\Big(\frac{\lambda}{1+\lambda}\big(x_1 - \vr(x)\big) + 1\Big)\,e^{\lambda \vr(x)},\qquad \vr(x) = x_1 + 1 - |x_2|.
$$
We calculate 
$$
\frac{\partial \VB}{\partial x_2} = \frac{\lambda^2}{1+\lambda}\,x_2 \,e^{\lambda \vr(x)},
$$ 
and see that this this derivative has the same sign as $x_2$. Therefore, for each $x_1$ fixed the restriction $\VB(x_1,\cdot)$ is not concave on $[-1,1]$. Therefore, $\UB > \VB$ in this case.
If we apply Theorem~\ref{TheoremForSubordination} and notice that the supremum in the right-hand side in~\eqref{FirstMajorization} is attained at $\delta=0$, we obtain an upper estimate for~$\UB$:
$$
\UB(x) \leq f(x_1) = -e^{\lambda x_1}.
$$

\paragraph{An example for Theorem~\ref{AngleTheorem}:~$p$-moments with~$p \geq 2$.} Consider the case~$f(t) = |t|^p$ for~$p \geq 2$. The~$\BMO$ case was considered in~\cite{SlavinVasyunin2012}, and the martingale case goes back to~\cite{Burkholder1988} (see also~\cite{Burkholder1991} and Section~$3.7$ in~\cite{Osekowski2012}). One may see that this function fulfills the hypothesis of Theorem~\ref{AngleTheorem} (the point~$0$ solves the balance equation by symmetry). The formula for the Bellman function is bulky and contains a truncated gamma function. We omit its presentation and pass to its consequence: the inequality
\eq{
\E|\psi_\infty|^p \leq \frac12 \Gamma(p+1), \quad \varphi_0 = \psi_0 = 0,
}
is sharp for the case~$|\varphi_\infty| = 1$ and~$\psi$ being a martingale transform of~$\varphi$ with a unimodular transforming sequence.

\section{Chordal domains}\label{S4}
\subsection{Construction of Bellman candidate}\label{s41}
We describe the foliation for~$\B$ called a chordal domain, see Section~$3.3$ in~\cite{ISVZ2018} for details. 
Let $a_0,a_1,b_0,b_1$ be real numbers such that $a_0 < a_1 \leq b_1 < b_0 \leq a_0 + 2$. Then, the points~$(a_0,a_0^2),$ $(a_1,a_1^2)$, $(b_1,b_1^2),$ and~$(b_0,b_0^2)$ lie on~$\FixedBoundary \Omega$ (see the left picture on Fig.~\ref{CHdVF}).
The chords~$\big[(a_0,a_0^2),(b_0,b_0^2)\big]$ and~$\big[(a_1,a_1^2),(b_1,b_1^2)\big]$ lie inside~$\Omega$ and the former separates the latter from the free boundary. We assume that the domain between them is foliated by some chords~$\big[(a(\ell),a^2(\ell)),(b(\ell),b^2(\ell))\big]$; it is convenient to parametrize them with the lengths of their projections onto the~$y_1$-axis, i.\,e.,~$b(\ell) - a(\ell) = \ell$. In other words, there are two functions~$a,b \colon [\ell_1,\ell_0]\to \R$, the first one decreasing and the second one increasing; here~$\ell_1 = b_1-a_1$ and~$\ell_0 = b_0-a_0$. Let~$\Ch([a_0,b_0],[a_1,b_1])$ be  the subdomain of~$\Omega$ bounded by~$\big[(a_0,a_0^2),(b_0,b_0^2)\big]$ and~$\big[(a_1,a_1^2),(b_1,b_1^2)\big]$ and equipped with the functions~$a$ and~$b$. We call this domain a \emph{chordal domain}. Define the function~$B$ {to be affine} on the chords:
\eq{\label{vallun}
B(y_1,y_2) = \alpha f(a(\ell)) + \beta f(b(\ell)),\qquad
(y_1,y_2) = \alpha (a(\ell),a^2(\ell)) + \beta (b(\ell),b^2(\ell)),\qquad \alpha+\beta =1,\ \alpha, \beta \geq 0.
}
From now on we will suppress the argument in our notation for the functions~$a$ and~$b$. We always assume these functions are differentiable and their derivatives do not vanish on~$(\ell_1,\ell_0)$.

\begin{St}[Proposition~$3.3.3$ in~\cite{ISVZ2018}]\label{LightChordalDomainCandidate}
Consider the domain~$\Ch([a_0,b_0],[a_1,b_1])$ with the functions~$a$ and~$b$ associated with it. Assume the following\textup:
\begin{itemize}
\item the functions~$a$ and~$b$ are differentiable on the interior of their domain\textup,~$a' < 0$\textup, and~$b' > 0$\textup;
\item for each~$\ell \in (\ell_1,\ell_0)$ the pair~$(a(\ell),b(\ell))$ satisfies the {\bf cup equation}
\eq{\label{urlun}
\frac{f(b) - f(a)}{b-a} = \frac{f'(a) + f'(b)}{2},\qquad \ell \in [\ell_1,\ell_0];
}
\item the differentials 
\eq{\label{DifDef}
\Dr = f''(b) - \frac{f'(b) - f'(a)}{b-a},\qquad \Dl = f''(a) - \frac{f'(b) - f'(a)}{b-a}
}
satisfy the inequalities $\Dr < 0$ and~$\Dl < 0$ on~$(\ell_1,\ell_0)$.
\end{itemize}
Then the function~$B$ defined by formula~\eqref{vallun} is a concave solution to the homogeneous Monge--Amp\`ere equation~\eqref{MongeAmpere} on~$\Ch([a_0,b_0],[a_1,b_1])$.
\end{St}

The name ``differentials'' for the expressions in~\eqref{DifDef} comes from differentiation of~\eqref{urlun}:{
\eq{
d\Big(\frac{f(b) - f(a)}{b-a} - \frac{f'(a) + f'(b)}{2}\Big) = \Dl da + \Dr db.
}
Therefore, if the pair~$(a,b)$ solves the cup equation~\eqref{urlun}, then}
\eq{\label{DifferentialsIdentity}
a'\Dl + b' \Dr = 0.
}
\begin{figure}[h!]
\includegraphics[height=3.3cm]{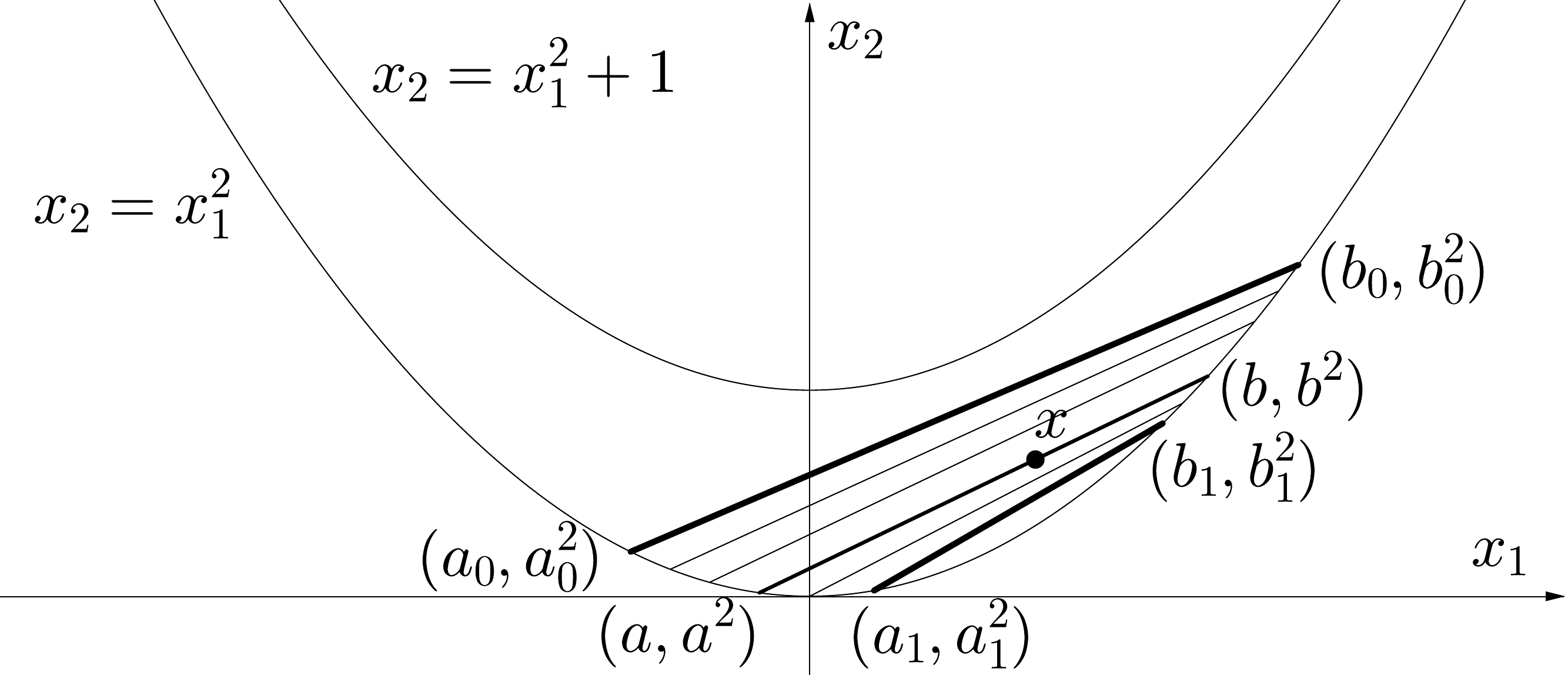}\qquad
\includegraphics[height=3.8cm]{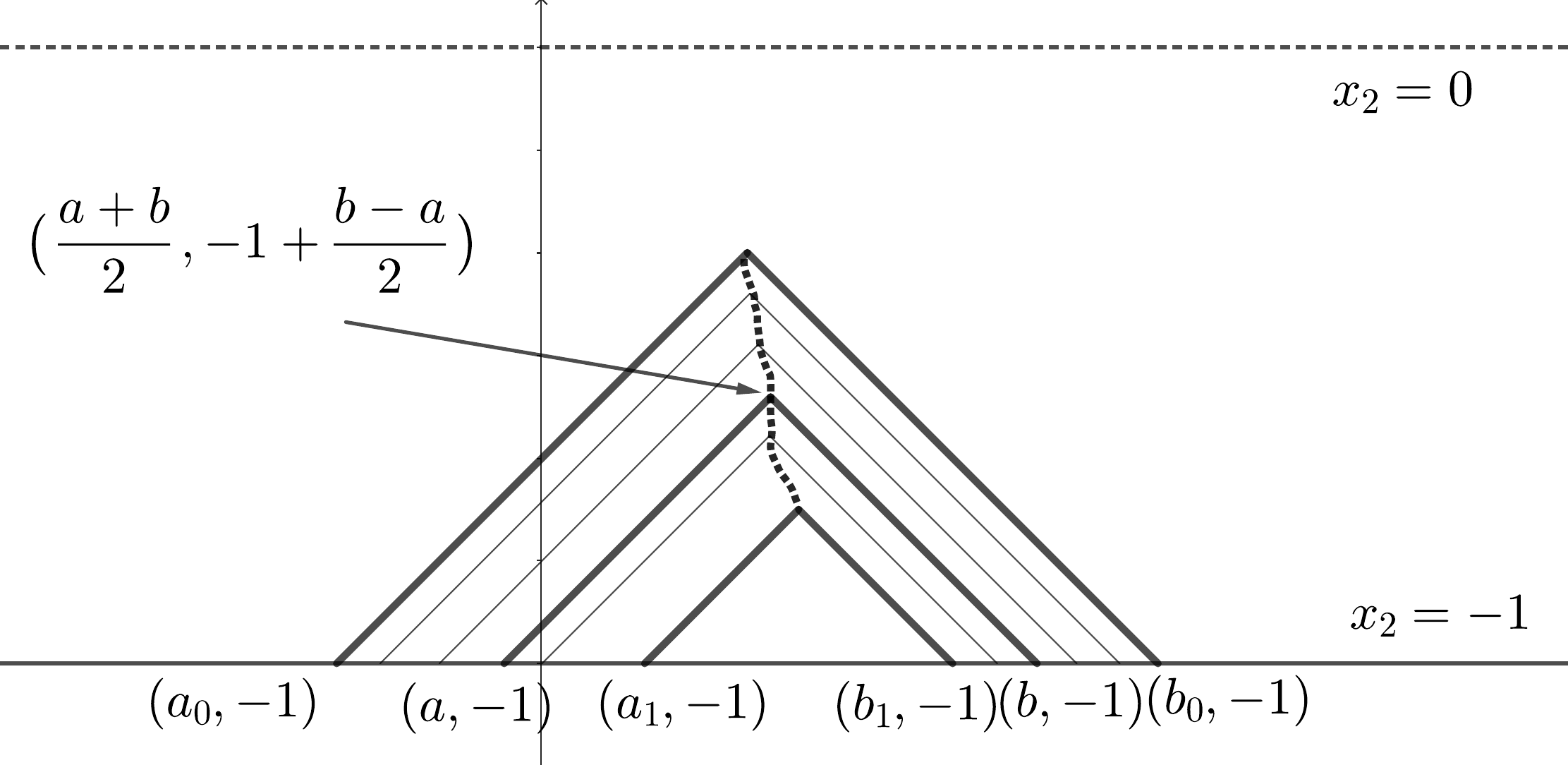}
\caption{A chordal domain and the corresponding vertical herringbone.}
\label{CHdVF}
\end{figure}

We define a function~$V$ on the subdomain
\eq{\label{SFir+}
\SFir([a_0,b_0],[a_1,b_1])\!=\!\SSet{x\in \Sigma}{\;\frac{b_1-a_1}{2}  - \Big|x_1\!-\!\frac{a_1+b_1}{2}\Big| \leq x_2\!+\!1 \leq \frac{b_0-a_0}{2}  - \Big|x_1\!-\!\frac{a_0+b_0}{2}\Big|}
} 

of~$\Sigma$ as follows: we consider the point
\eq{
P(\ell) = \Big(\frac{a+b}{2}, -1+\frac{b-a}{2}\Big),\qquad \ell \in [\ell_1,\ell_0],
}
and assume that~$V(P(\ell)) = A(\ell)$ (where~$A$ is an unknown function of a single variable) and that~$V$ is affine on the segments~$[(a,-1),P]$ and~$[P,(b,-1)]$. In other words, {we assume that}
\eq{\label{FirTreeFunction}
V(x) = 
\begin{cases}
f(b) + \frac{2(x_2+1)}{\ell}(A(\ell) - f(b)),\quad &x_1 \geq \frac{a+b}{2};\\
f(a) + \frac{2(x_2+1)}{\ell}(A(\ell) - f(a)),\quad &x_1 \leq \frac{a+b}{2},
\end{cases}
\qquad x\in \SFir([a_0,b_0],[a_1,b_1]).
}
We will call this foliation of~$\SFir$ a \emph{vertical herringbone}\footnote{{Recall the folitaions~$\SRt$ and~$\SLt$ drawn on Fig.~\ref{TangHorRight} and Fig.~\ref{AngleEnvelope} are called horizontal herringbones.}}. The dependence of~$\ell$ (and, thus,~$a$ and~$b$) on~$x$ is a little bit tricky. We have either~$a=a(\ell(x)) = x_1-x_2 - 1$ and then~$b$ is defined via~\eqref{urlun} or~$b=b(\ell(x)) = x_1+x_2+1$ and~$a$ is defined via~\eqref{urlun}. Note that, in particular, either~$x_1-x_2-1\in[a_0,a_1]$ or~$x_1+x_2+1 \in [b_1,b_0]$. The choice between these two possibilities is made according to the geometry of the picture (see Fig.~\ref{CHdVF}): the first possibility occurs when~$x_1 \leq (a+b)/2$.

\begin{St}\label{VerticalFir}
Consider the domain~$\SFir([a_0,b_0],[a_1,b_1])$ with the functions~$a,b\colon [\ell_1,\ell_0]\to \R$. Assume these functions are differentiable and their derivatives do not vanish on~$(\ell_1,\ell_0)$\textup; here~$a$ decreases and~$b$ increases.
Let~$V$ be given by~\eqref{FirTreeFunction}\textup, where~$A\colon [\ell_1,\ell_0]\to \R$ satisfies the differential equation
\eq{\label{ODEForA}
A'(\ell)\ell - A(\ell) = -\frac{f(a) \Dl + f(b) \Dr}{\Dr+\Dl}.
}
If the functions~$a$ and~$b$ satisfy the cup equation~\eqref{urlun} and the inequalities~$\Dr < 0$ and~$\Dl < 0$\textup, then~$V$ is a diagonally concave solution to~\eqref{StrangeEquation}.
\end{St} 
We note that a reverse implication is also true: the~$C^1$\!-smoothness of~$V$ given by~\eqref{FirTreeFunction} implies~\eqref{ODEForA} while its diagonal concavity yields the cup equation and that the differentials have negative signs. We will not prove this assertion.
\begin{proof}
We will be working mostly with the right `half'~$x_1 > (a+b)/2$ of the domain~$\SFir([a_0,b_0],[a_1,b_1])$. We have~$b = 1+x_1+x_2$ there, and~$a$ is defined by~\eqref{urlun}. We will also need the formulas
\eq{\label{eq58}
b_{x_1}'=b_{x_2}' = 1;\quad a_{x_1}'=a_{x_2}'= -\frac{\Dr}{\Dl};\quad \ell_{x_1}' = \ell_{x_2}'= 1+\frac{\Dr}{\Dl},\qquad x_1 \geq \frac{a+b}{2};
}
they follow from~\eqref{DifferentialsIdentity}.

We verify the~$C^1$-smoothness of the function~$V$. This is equivalent to the coincidence on the interface~$\set{P(\ell)}{\ell \in [\ell_1,\ell_0]}$ of the directional derivatives of the right and left `halves' of the function. Since the restriction of~$V$ to each of the `halves' of the domain is~$C^1$-smooth up to the boundary, and~$V$ is continuous, it suffices to verify the coincidence of the directional derivative along a single direction that is transversal to the common border between the two `halves'. We choose the direction~$e = (1,1)$ and compute the corresponding directional derivative using~\eqref{eq58}:
\mlt{\label{Bx+yOntheright}
\frac{\partial V}{\partial e} = \frac{\partial V}{\partial x_1}+ \frac{\partial V}{\partial x_2} = 2f'(b) + \frac{2}{\ell}(A(\ell) - f(b)) - \frac{2(x_2+1)}{\ell^2}\cdot 2\big(1+ \frac{\Dr}{\Dl}\big)(A(\ell) - f(b))\\ + \frac{2(x_2+1)}{\ell} \Big(A'(\ell)\cdot 2\big(1+ \frac{\Dr}{\Dl}\big) - 2f'(b)\Big),\quad x_1 \geq \frac{a+b}{2}.
} 
When~$x_1 \leq (a+b)/2$, by linearity of~$V$ along~$e$,
\eq{\label{Bx+yOntheleft}
\frac{\partial V}{\partial e}  = \frac{2}{\ell}(A(\ell) - f(a)).
}
The~$C^1$-smoothness condition means that~\eqref{Bx+yOntheleft} equals the value of~\eqref{Bx+yOntheright} when~$x_2 = -1+\ell/2$. This is rewritten as
\mlt{
\frac{2}{\ell}(A(\ell) - f(a)) = 2f'(b) + \frac{2}{\ell}(A(\ell) - f(b)) - \frac{2}{\ell}\big(1+ \frac{\Dr}{\Dl}\big)(A(\ell) - f(b))\\ + 2\Big(A'(\ell)\cdot \big(1+ \frac{\Dr}{\Dl}\big) - f'(b)\Big),
}
which algebraically reduces to~\eqref{ODEForA}. 

It remains to prove~$V$ is diagonally concave on its domain. Since we have proved the~$C^1$-smoothness, it suffices to verify the claim on each of the `halves' of the domain individually. By symmetry, we may concentrate on the right `half'. We need to check that the second derivative of~$V$ along~$e$ is non-positive. By Lemma~\ref{affineLemma}, it suffices to prove the inequalities
\alg{
\label{ineq1}&\frac{\partial^2 V}{\partial e^2}\Big|_{x_2 = -1+\ell/2} \leq 0,\\
\label{ineq2}&\frac{\partial^2 V}{\partial e^2}\Big|_{x_2 = -1} \leq 0.
}
\paragraph{Verification of~\eqref{ineq1}.} In fact, we will prove that this inequality turns into equality. Fix~$\ell$. We see that~$\frac{\partial V}{\partial e}$ is affine as a function of~$x_2$. By the already proved~$C^1$-smoothness,  the value of this affine function equals~\eqref{Bx+yOntheleft} when~$x_2 = -1 + \ell/2$. Thus,
\eq{\label{AffineDerivative}
\frac{\partial V}{\partial e} = k(a,b)(x_2 + 1 - \ell/2) +  \frac{2}{\ell}(A(\ell) - f(a)),
}
where
\eq{\label{expressionFork}
k(a,b) = -\frac{4}{\ell^2}f(a) + \frac{4}{\ell^2}f(b) - \frac{4}{\ell}f'(b);
}
this follows from~\eqref{Bx+yOntheright} by substituting~$x_2=-1$.

Thus,~\eqref{AffineDerivative} implies
\eq{
\frac{\partial^2 V}{\partial e^2}\Big|_{x_2 = -1 + \ell/2} = -\frac{\Dr}{\Dl}k(a,b) + \Big(\frac{2}{\ell}(A(\ell) - f(a))\Big)_{e}^\prime.
}
This expression is identically zero by~\eqref{urlun},~\eqref{ODEForA}, and~\eqref{expressionFork}.

\paragraph{Verification of~\eqref{ineq2}.} When~$x_2=-1$, we differentiate~\eqref{Bx+yOntheright}, plug~$x_2 = -1$, and obtain
\mlt{
\frac{\partial^2 V}{\partial e^2}\Big|_{x_2 = -1} 
= 4f''(b) + 2\Big(\frac{A(\ell) -f(b)}{\ell}\Big)^\prime_{e}  - \frac{4}{\ell^2}\Big(1+ \frac{\Dr}{\Dl}\Big)\big(A(\ell) - f(b)\big) \\+ \frac{4}{\ell}\Big(A'(\ell)\Big(1+\frac{\Dr}{\Dl}\Big) - f'(b)\Big).
}
Using~\eqref{ODEForA} and algebraic transformations, this reduces to
\eq{
4f''(b) + \frac{8}{\ell^2}(f(b) - f(a)) - \frac{8}{\ell}f'(b).
}
By~\eqref{urlun}, this expression equals~$4\Dr$, which is non-positive.
\end{proof}
{
\begin{Rem}
We have described the foliation and the function~$V$ in~$\SFir([a_0,b_0],[a_1,b_1])$. Since~$V$ is symmetric with respect to the~$x_1$-axis, it behaves exactly the same way in the symmetric domain
\eq{\label{SFir-}
\SSet{x\in \Sigma}{\;\frac{b_1-a_1}{2}  - \Big|x_1 -\frac{a_1+b_1}{2}\Big| \leq 1- x_2 \leq \frac{b_0-a_0}{2}  - \Big|x_1 -\frac{a_0+b_0}{2}\Big|}.
}
From now on\textup, when we write~$\SFir([a_0,b_0],[a_1,b_1])$\textup, we mean the union of~\eqref{SFir+} and~\eqref{SFir-}. 
\end{Rem}
}

\subsection{Gluing chords with tangents}\label{s42}
We start with the observation that unlike the situation with the chordal domain in~$\Omega$, there is no simple formula like~\eqref{vallun} for the function~$V$ in~$\SFir$. The values of~$A$, which is a solution of~\eqref{ODEForA}, are expressed as
\eq{\label{ValueTransmissionAlongVerticalFir}
A(\ell) = \frac{\ell}{\ell_0}A(\ell_0) + \ell\int\limits_{\ell}^{\ell_0} \frac{f(a)\Dl + f(b)\Dr}{\Dl + \Dr}\frac{d l}{l^2};
}
the integrand is interpreted as a function of~$l$. {To compute~$A(\ell)$ via~\eqref{ValueTransmissionAlongVerticalFir}, we need to know all the values~$f(a)$ and~$f(b)$ for larger~$\ell$.} The value~$A(\ell_0)$ is a free parameter. In the case~$\ell_0 = 2$, the domain~$\SFir$ meets the symmetric domain at the point~$P(2)$ (see Fig.~\ref{CTFigure}). If the resulting function is {continuous}, then~$B(P(2)) = (f(a_0) + f(b_0))/2$. In other words, in this case
\eq{\label{BoundaryValueSymmetricVerticalFir}
A(2) = \frac{f(a_0) + f(b_0)}{2}.
}

\begin{figure}[h!]
\centering
\includegraphics[height=6cm]{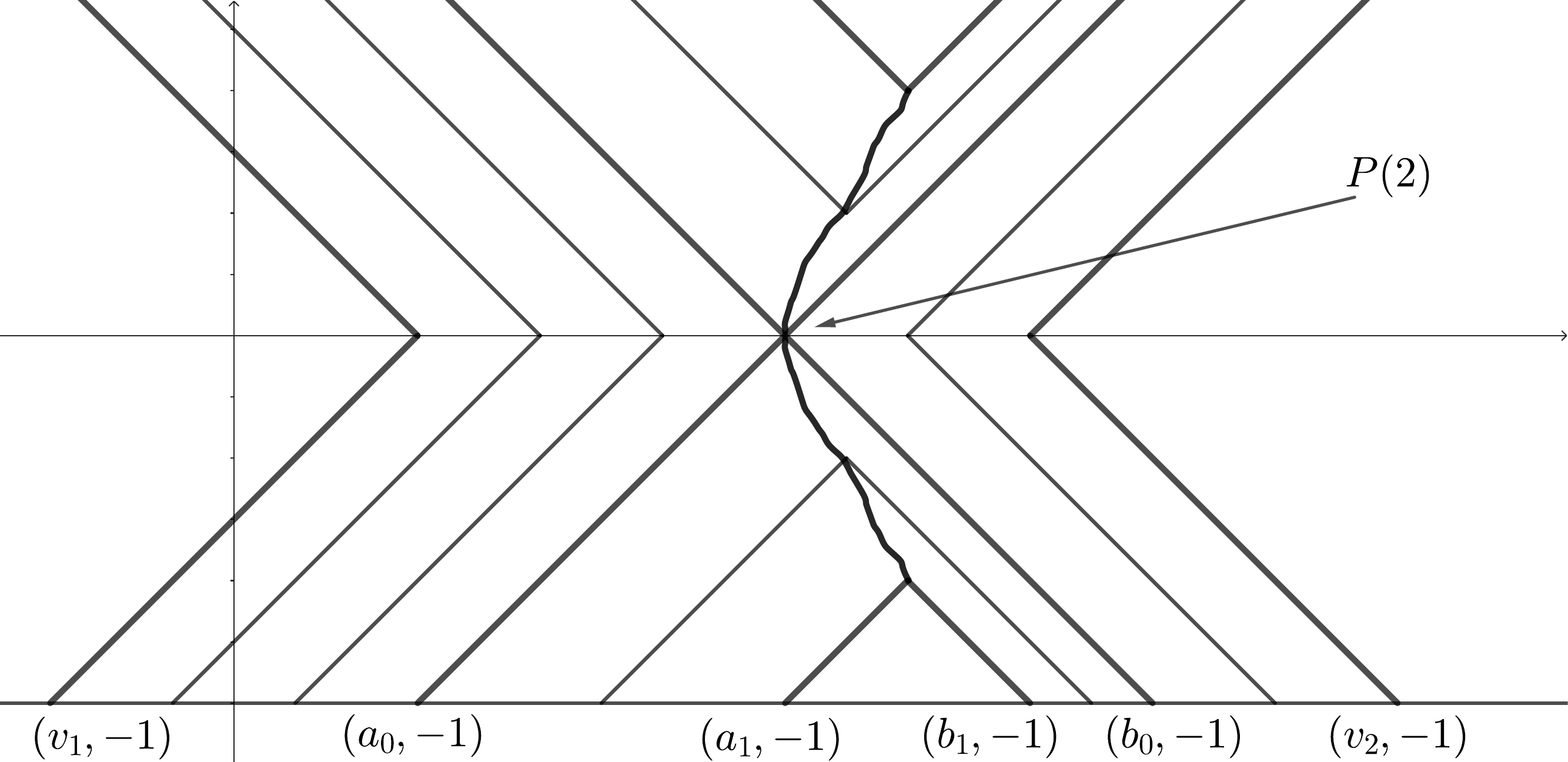}
\caption{A vertical herringbone meets a horizontal herringbone.}
\label{CTFigure}
\end{figure}

It appears that this condition is sufficient to concatenate the functions constructed in Propositions~\ref{HorizontalFurRight} and~\ref{HorizontalFurLift} with the one constructed in Proposition~\ref{VerticalFir}.

\begin{St}\label{ChordalDomainProposition}
Let~$b_0 - a_0 = 2$. Consider the domains~$\SFir([a_0,b_0], [a_1,b_1])$\textup, $\SRt(b_0,v_2)$\textup, and~$\SLt(v_1,a_0)$. 
We define the function~$V$ on~$\SFir([a_0,b_0], [a_1,b_1])$  by formulas~\eqref{FirTreeFunction}\textup, 
\eqref{ValueTransmissionAlongVerticalFir}\textup, and~\eqref{BoundaryValueSymmetricVerticalFir}. 
For the values on~$\SRt(b_0,v_2)$ and~$\SLt(v_1,a_0)$\textup, we use the initial data
\eq{\label{ChordsTangentsGluing}
\Mlt(a_0) = \Mrt(b_0) = \frac{f(b_0) - f(a_0)}{2}
}
and define~$V$ {by~\eqref{difeq} and by~\eqref{BFRT} on~$\SRt(b_0,v_2)$ and  by~\eqref{difeq2} and by~\eqref{BFLT} 
on~$\SLt(v_1,a_0)$}. If the function~$f$ falls under the scope of Propositions~\textup{\ref{HorizontalFurRight}, 
\ref{HorizontalFurLift},} and~\emph{\ref{VerticalFir}}\textup, then~$V$ is a~$C^1$\!-smooth diagonally concave function solving~\eqref{StrangeEquation} on  each of  these domains.
\end{St}

\begin{proof}
It suffices to verify that~$V$ is a~$C^1$-smooth function. The conditions~\eqref{BoundaryValueSymmetricVerticalFir} and~\eqref{ChordsTangentsGluing} guarantee the continuity of~$V$. By symmetry, we may restrict our attention to the~$C^1$-smoothness in a neighborhood of the segment~$[P(2), (b_0,-1)]$. By Lemma~\ref{affineLemma}, it is enough to prove differentiability in a neighborhood of the two endpoints. As usual, we may prove coincidence of directional derivatives in a transversal direction. For the point~$(b_0,-1)$, we choose derivatives with respect to~$(1,0)$: they both coincide with~$f'(b_0)$. For the point~$P(2)$, we choose the direction~$(1,1)$: both derivatives coincide with~$\frac{f(b_0) - f(a_0)}{2}$ by construction.
\end{proof}

\begin{Rem}\label{RemarkChords}
It is shown in  Propositions~$3.3.7$ and~$3.3.8$ of~\textup{\cite{ISVZ2018}} that exactly the same 
condition~\eqref{ChordsTangentsGluing} guarantees the~$C^1$\!-smoothness of the concatenation~$B$ of the functions 
on the domains~$\Lt(u_1,a_0)$\textup, $\Rt(b_0,u_2)$\textup, and~$\Ch([a_0,b_0], [a_1,b_1])$ described in 
Propositions~\textup{\ref{321ISVZ}, \ref{322ISVZ},} and~\textup{\ref{LightChordalDomainCandidate}}. 
In particular,~$B(x_1,x_1^2 + 1) = V(x_1,0)$ when~$x_1 \in (u_1+1,u_2-1)$.
\end{Rem}

\subsection{Example}\label{s43}
\paragraph{The case~$f(t) = |t|^p$ with~$p \in [1,2]$.} The function~$\B$ for this boundary data was found in~\cite{SlavinVasyunin2012}. Note that this function does not fulfill Condition~\ref{reg}. One may overcome this peculiarity by approximation. The foliation for~$\B$ consists of the full chordal domain~$\Ch([-1,1],[0,0])$ and the tangent domains~$\Lt(-\infty,-1)$ and~$\Rt(1,\infty)$. The chordal domain is symmetric and the functions~$a$ and~$b$ have very simple form:
\eq{
a(\ell) = -\frac{\ell}{2}\qquad b(\ell) = \frac{\ell}{2}.
}
According to Proposition~\ref{ChordalDomainProposition}, we may describe the function~$\VB$ as follows. The values of~$\VB$ on~$\SRt(1,\infty)$ are 
\eq{
\VB(x) = v^p +(x_1 - v)e^{-v}\int\limits_{1}^vpt^{p-1}e^t\,dt,\qquad v=\vl(x) =  x_1 + 1 - |x_2|.
}
For the values on the vertical herringbone, we use~\eqref{ValueTransmissionAlongVerticalFir}:
\eq{
\Dr = \Dl = p(p-2)(\ell/2)^{p-2};\qquad A(\ell) = \frac{\ell}{2} + \ell \int\limits_{\ell}^2 (l/2)^p\,\frac{dl}{l^2} = \frac{\ell}{2} + \frac{2^{-p}\ell}{p-1}\Big(2^{p-1} - \ell^{p-1}\Big).
}
Therefore,
\eq{
A(\ell) - f(b) = \frac{p}{p-1}\Big(\frac{\ell}{2} - \Big(\frac{\ell}{2}\Big)^{p}\Big),
}
and, according to~\eqref{FirTreeFunction},
\eq{
\VB(x) = \Big(\frac{\ell}{2}\Big)^p \!\!+ \frac{p}{p-1}(1-|x_2|)\Big(1-\Big(\frac{\ell}{2}\Big)^{p-1}\Big),\quad \ell = 2 + 2|x_1| - 2|x_2|,\quad x\in \SFir([-1,1],[0,0]).
}

\section{Other linearity domains}\label{S5}
\subsection{Trolleybuses}\label{s51}
We describe a linearity domain for~$\B$ called a \emph{trolleybus}. It has two points~$(a_0,a_0^2)$ and~$(b_0,b_0^2)$ on the lower boundary, a chordal domain below it and two equally oriented tangent domains adjacent to it from the sides. If both tangents are left, it is a left trolleybus. If both tangents are right, it is a right trolleybus, see Fig.~\ref{RightTrolleyFigure} for visualization. We denote the trolleybuses by~$\LTroll(a_0,b_0)$ and~$\RTroll(a_0,b_0)$. One may see that the values of the affine function inside the trolleybus are completely defined by the values~$f(a_0)$,~$f(b_0)$,~$f'(a_0)$, and~$f'(b_0)$, provided the said affine function might be extended to a~$C^1$-smooth function satisfying the boundary conditions~\eqref{BCForBMO} on a neighborhood of the trolleybus. The corresponding system of linear equations is overdetermined (there are three coefficients of an affine function and four equations). The compatibility conditions are given by the cup equation~\eqref{urlun} for~$a_0$ and~$b_0$. More precisely, if
\eq{\label{InTrolley}
B(y) = \beta_2y_2 + \beta_1y_1 + \beta_0,\qquad y\in \RTroll(a_0,b_0),
}
then, as formula~$(3.4.10)$ in~\cite{ISVZ2018} says,
\begin{equation}\label{coef}
\begin{split}
&\beta_0=\frac{b_0 f(a_0)-a_0 f(b_0)}{b_0-a_0}+\frac{1}{2}a_0 b_0\av{f''}{[a_0,b_0]};\\
&\beta_1 = \av{f'}{[a_0,b_0]} - \frac{1}{2}(b_0+a_0)\av{f''}{[a_0,b_0]};\\
&\beta_2 = \frac{1}{2}\av{f''}{[a_0,b_0]}.
\end{split}
\end{equation} 

\begin{figure}[h!] 
\centering
\includegraphics[height=6cm]{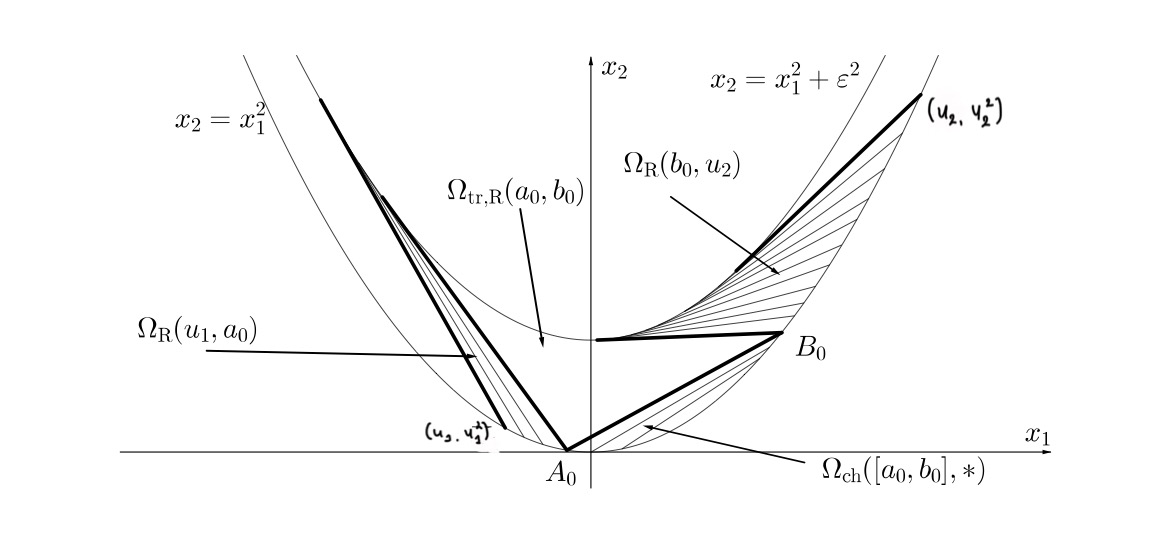}
\caption{A right trolleybus}
\label{RightTrolleyFigure}
\end{figure}

We cite Proposition~$3.4.12$ from~\cite{ISVZ2018} that describes the conditions for a~$C^1$-smooth concatenation of the affine function in a trolleybus with the functions on adjacent figures.

\begin{St}[Proposition~$3.4.12$ in~\cite{ISVZ2018}]\label{RTrollProp}
Let $u_1<a_0<b_0<u_2$ and let $b_0-a_0\le 2$. Let the function~$B$ coincide with the functions given by 
Proposition~\textup{\ref{321ISVZ}} on~$\Rt(u_1,a_0)$ and $\Rt(b_0,u_2)$\textup, by 
Proposition~\textup{\ref{LightChordalDomainCandidate}} on{\footnote{The following notation is borrowed from~\cite{ISVZ2018}. It is used when one of the border chords of a chordal domain is unimportant.}}~$\Ch([a_0,b_0],*)$\textup, and by 
formulas~\eqref{InTrolley}\textup, \eqref{coef} in~$\RTroll(a_0,b_0)$. Suppose that 
\begin{align}
\Mrt(a_0) = f'(a_0) - \av{f''}{[a_0,b_0]}\label{mrta0};\\ 
\Mrt(b_0) = f'(b_0) - \av{f''}{[a_0,b_0]}\label{mrtb0}.
\end{align}
Then\textup,~$B$ is a~$C^1$\!-smooth locally concave function on
\eq{
\Rt(u_1,a_0) \cup \RTroll(a_0,b_0) \cup \Ch([a_0,b_0],*) \cup \Rt(b_0,u_2)
}
{solving the homogeneous Monge--Amp\`ere equation~\eqref{MongeAmpere} on each of the domains of the partition.}
\end{St}
In the case of a left trolleybus, equations~\eqref{mrta0} and~\eqref{mrtb0} should be replaced with
\begin{align}
\Mlt(a_0) = f'(a_0) + \av{f''}{[a_0,b_0]}\label{mlta0};\\ 
\Mlt(b_0) = f'(b_0) + \av{f''}{[a_0,b_0]}\label{mltb0}.
\end{align}

Now we pass to the description of a similar foliation in~$\Sigma$. It will be called a \emph{corner}, see Fig.~\ref{RightCornerFigure} for visualization. A corner might be right or left. We will not provide formulas and simply describe a right corner as a domain of bilinearity lying between~$\SFir([a_0,b_0],[a_1,b_1])$,~$\SRt(v_1,a_0)$, and~$\SRt(b_0,v_2)$. Here~$v_1 < a_0 < b_0 < v_2$ and~$b_0 - a_0 \leq 2$. We denote it by~$\RCorn(a_0,b_0)$. We assume that the function~$V$ is given by Proposition~\ref{HorizontalFurRight} inside the domains~$\SRt(v_1,a_0)$ and~$\SRt(b_0,v_2)$ and by Proposition~\ref{VerticalFir} inside~$\SFir([a_0,b_0],[a_1,b_1])$. We also assume~$V$ is bilinear inside the corner:
\eq{\label{InCorner}
V(x) = \alpha_{11}(x_1^2 - x_2^2) + \alpha_1x_1 + \alpha_0, \qquad x\in \RCorn(a_0,b_0).
}
As it happened before, there is no term with~$x_2$ since~$V$ is symmetric with respect to~$x_2$.

\begin{figure}[h!] 
\centering
\includegraphics[height=6cm]{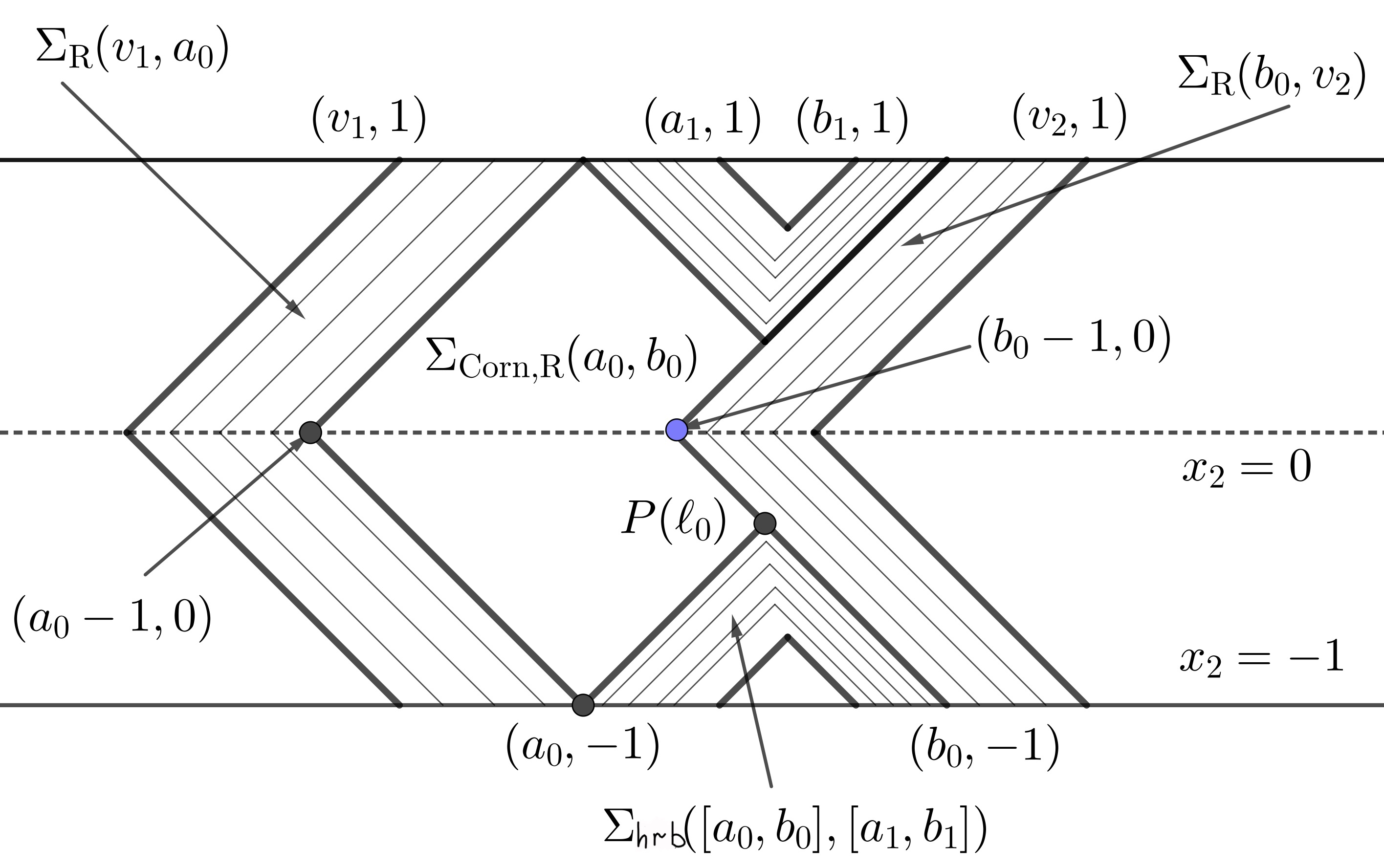}
\caption{A right corner}
\label{RightCornerFigure}
\end{figure}

The continuity of~$V$ on the common borders of~$\RCorn(a_0,b_0)$ with~$\SRt(v_1,a_0)$ and~$\SRt(b_0,v_2)$ yields
\alg{
\alpha_{11}(a_0^2 - 1) + \alpha_1 a_0 + \alpha_0 = f(a_0);\\
\alpha_{11}(b_0^2 - 1) + \alpha_1 b_0 + \alpha_0 = f(b_0);\\
2\alpha_{11}(a_0 - 1) + \alpha_1 = \Mrt(a_0);\\
2\alpha_{11}(b_0-1) + \alpha_1 = \Mrt(b_0).
}
Note that this linear system is overdetermined and the compatibility condition is exactly the cup equation~\eqref{urlun}. Now we assume~\eqref{mrta0} and~\eqref{mrtb0} hold true. Solving the system with these identities in mind, we obtain
\begin{equation}\label{coef2}
\begin{split}
&\alpha_0=\frac{b_0 f(a_0)-a_0 f(b_0)}{b_0-a_0}+\frac{1}{2}a_0 b_0\av{f''}{[a_0,b_0]} + \alpha_{11};\\
&\alpha_1 = \av{f'}{[a_0,b_0]} - \frac{1}{2}(b_0+a_0)\av{f''}{[a_0,b_0]};\\
&\alpha_{11} = \frac{1}{2}\av{f''}{[a_0,b_0]}.
\end{split}
\end{equation} 

\begin{St}\label{TrolleyToCorner}
Let $v_1<a_0<b_0<v_2$ and let $b_0-a_0\le 2$. Let the function~$V$ coincide with the functions given by 
Proposition~\textup{\ref{HorizontalFurRight}} on~$\SRt(v_1,a_0)$ and $\SRt(b_0,v_2)$\textup, by 
Proposition~\textup{\ref{VerticalFir}} on~$\SFir([a_0,b_0],*)$\textup, and by formulas~\eqref{InCorner}\textup, 
\eqref{coef2} in~$\RCorn(a_0,b_0)$. Assume~\eqref{mrta0} and~\eqref{mrtb0}.
Then\textup, $V$ is a~$C^1$\!-smooth diagonally concave function on
\eq{\label{PartitionAroundCorner}
\SRt(v_1,a_0) \cup \RCorn(a_0,b_0) \cup \SFir([a_0,b_0],*) \cup \SRt(b_0,v_2)
}
solving~\eqref{StrangeEquation} on each of domain of this partition.
\end{St}

\begin{proof}
We have already verified that~$V$ is a continuous function, diagonally concave on each of the domains in the partition~\eqref{PartitionAroundCorner}. Thus, it remains to verify that~$V$ is a~$C^1$-smooth function. By {symmetry and} Lemma~\ref{affineLemma}, we may prove the~$C^1$-continuity of the concatenation at the points~$(a_0-1,0)$,~$(a_0,-1)$,~$P(\ell_0)$, and~$(b_0-1,0)$.

Let us consider the case of the point~$(a_0,-1)$. Let~$V|_{\SRt(v_1,a_0)}$ be~$V_1$ and let~$V|_{\RCorn(a_0,b_0)}$ be~$V_2$. We need to check that
\eq{\label{GradientCoincidence}
\nabla V_1(a_0,-1) = \nabla V_2(a_0,-1).
}
We know that both functions are~$C^1$-smooth on their domains. Moreover, the directional derivatives with respect to~$(1,-1)$ coincide since both functions are affine in this direction. It suffices to find another direction such that the corresponding derivatives of~$V_1$ and~$V_2$ coincide. We choose~$(1,0)$. By smoothness of~$V_1$, its derivative w.r.t.~$x_1$ at~$(a_0,-1)$ equals~$f'(a_0)$. By~\eqref{InCorner},
\eq{
\frac{\partial V_2}{\partial x_1}(a_0,-1) = 2\alpha_{11}a_0 + \alpha_1 \Eeqref{coef2} f'(a_0).
}
Thus, we have proved~\eqref{GradientCoincidence}.

The cases {of the points~$(a_0-1,0)$,~$P(\ell_0)$, and~$(b_0-1,0)$} are completely similar.
\end{proof}

\begin{Rem}\label{RemarkTrolleybus}
By formulas~\eqref{InTrolley} and~\eqref{InCorner}\textup,
\alg{
B(x_1,x_1^2 + 1) &= \beta_2(x_1^2 + 1) + \beta_1x_1 + \beta_0;\\
V(x_1,0) &= \alpha_{11}x_1^2 + \alpha_1x_1 + \alpha_0,
}
provided~$(x_1,0)\in \RCorn(a_0,b_0)$ or~$(x_1,0)\in \LCorn(a_0,b_0)$. Thus, the assertion of 
Theorem~\textup{\ref{TheoremForDiagonal}} in the case of a trolleybus and the corresponding corner reduces to
\eq{
\alpha_{11} = \beta_1,\quad \alpha_1 = \beta_1,\quad \alpha_0 = \beta_0 + \beta_2.
}
This follows directly from~\eqref{coef} and~\eqref{coef2}.
\end{Rem}

\begin{Rem}
There is yet another type of  linearity domains with two points on the lower boundary of~$\Omega$ called {\bf birdie}. A birdie is a linearity domain with two points on the fixed boundary bounded by a right tangent from the left and by a left tangent from the right. We will not provide a separate study for the transference of functions from birdies to~$\Sigma$\textup, because\textup, as formula~$(3.4.19)$ in~\textup{\cite{ISVZ2018}} says\textup, a birdie may be treated as a union of a trolleybus and an angle\textup; there are two ways to represent a birdie in this way. We transfer the function from the trolleybus to the corner according to Proposition~\textup{\ref{TrolleyToCorner},} transfer the function from the angle to the square according to Proposition~\textup{\ref{TransferToSquare},} and obtain the bilinear function on the union of the corner and the square due to~$C^1$\!-smoothness.
\end{Rem}

\subsection{Other linearity domains}\label{s52}
We start with a survey of those linearity domains of minimal locally concave functions on~$\Omega$ that have more than two points on the lower boundary. Let~$\mathfrak{L}$ be such a domain. Let~$B$ be an affine function on~$\El$:
\eq{\label{AffineFunctionInLinearityDomain}
B(x) = \beta_0 + \beta_1x_1 + \beta_2x_2,\qquad x\in \El.
}
Assume that the intersection of~$\El$ with the lower parabola consists of a finite number of arcs. Call them~$\{\Arc_i\}_{i=1}^n$ 
and assume that
\eq{
\Arc_i = \set{(t,t^2)}{t\in \arc_i},\qquad i=1,2,\ldots, n,
}
where~$\arc_i = [\arcl_i,\arcr_i]$, and $\arcr_i < \arcl_{i+1}$ for $i=1,\dots,n-1$.

Note that~$\arcl_{i+1} - \arcr_i < 2$. We assume~$\El$ intersects the upper boundary,~i.\,e.,~$\arcr_n - \arcl_1 > 2$. Since~$B$ satisfies the boundary condition~\eqref{BCForBMO},
\eq{\label{ValuesGluing}
f(t) = \beta_0 + \beta_1t+ \beta_2t^2,\qquad t\in \bigcup_{i=1}^n\arc_i,
}
and, by~$C^1$-smoothness,
\eq{\label{DerivativesGluing}
\beta_1 + 2t\beta_2 - f'(t) = 0,\qquad t\in \bigcup_{i=1}^n\arc_i.
}
Consider a chord with the endpoints $\big(\arcr_i,(\arcr_i)^2\big)$ and $\big(\arcl_{i+1},(\arcl_{i+1})^2\big)$. These two points 
satisfy the cup equation~\eqref{urlun}, provided there is a chordal domain adjacent to~$\El$ from below along this chord. Lemma~$3.4.17$ in~\cite{ISVZ2018} says that in such a case (we also assume~\eqref{ValuesGluing} and~\eqref{DerivativesGluing}) any two points in~$\cup_{i=1}^n\Arc_i$ satisfy the cup equation. One may prove a reverse implication: if any two points in~$\cup_{i=1}^n\Arc_i$ satisfy the cup equation, then there exist~$\beta_0,\beta_1$, and~$\beta_2$ such that~\eqref{ValuesGluing} and~\eqref{DerivativesGluing} hold true. In such a case, the function~$B$ constructed by~\eqref{AffineFunctionInLinearityDomain} in~$\El$ and the functions constructed by~\eqref{vallun} in the chordal domains~$\Ch([\arcr_i,\arcl_{i+1}],*)$,~$i=1,2,\ldots, n-1$, concatenate in a~$C^1$-smooth way. We note that the coefficients~$\beta_0, \beta_1,$ and~$\beta_2$ are defined by~\eqref{coef}, where~$a_0$ and~$b_0$ are replaced with any two distinct points in~$\cup_i\arc_i$ (the result does not depend on this particular choice).

Consider now the point~$\big(\arcr_n,(\arcr_n)^2\big)$. There is a tangent (left or right) that passes through this point and separates~$\El$ from a tangent domain (either left or right) lying on the right of this linearity domain. If this is a right tangent, then the condition of~$C^1$-smooth concatenation between~$B$ and the function given by~\eqref{linearity} and~\eqref{difeq} in~$\Rt(\arcr_n,*)$ is
\eq{\label{MultiTangentR}
\Mrt(\arcr_n) = f'(\arcr_n) - \av{f''}{[\arcl_1,\arcr_n]},
} 
the same as for the right trolleybus~\eqref{mrtb0} (the point~$\arcl_1$ in the formula above may be replaced with any other point in~$\cup_{i=1}^n\arc_i$ by Lemma~$3.4.16$ in~\cite{ISVZ2018}). If the tangent passing through~$\arcr_n$ is left, then the condition of~$C^1$-smooth concatenation between~$B$ and the function given by~\eqref{linearity} and~\eqref{difeq2} in~$\Lt(\arcr_n,*)$ is
\eq{\label{MultiTangentL}
\Mlt(\arcr_n) = f'(\arcr_n) + \av{f''}{[\arcr_1,\arcr_n]},
} 
the same as for the left trolleybus~\eqref{mltb0}. The same rule applies to the left boundary of~$\El$ (it should be treated as the boundary of the corresponding trolleybus). A linearity domain with equally oriented border tangents and at least three points on the lower boundary is called a \emph{multitrolleybus} (right or left); see Fig.~$3.12$ in~\cite{ISVZ2018} for visualization. If a linearity domain with more than two points on the lower boundary has left tangent on the left and right on the right, then it is called a \emph{multicup}; see Fig.~$3.13$ in~\cite{ISVZ2018}. In the remaining case, it is called a \emph{multibirdie}; see Fig.~$3.14$ in~\cite{ISVZ2018}.

Now we transfer these constructions to~$\Sigma$. Instead of writing down a formal description of a bilinearity domain~$\BEl$, we will describe its surroundings. The chordal domains~$\Ch([\arcr_i,\arcl_{i+1}],*)$,~$i=1,2,\ldots, n-1$, are transferred to vertical herringbones~$\SFir([\arcr_i,\arcl_{i+1}],*)$,~$i=1,2,\ldots, n-1$, see Subsection~\ref{s41}. As for the tangent domains that bound~$\El$ from the sides, we transfer them to the corresponding horizontal herringbones~$\SRt$ and~$\SLt$. For example, a multicup drawn on Fig.~$3.13$ in~\cite{ISVZ2018} is transferred to the domain drawn on Fig.~\ref{MCupFigure}. The values of the corresponding bilinear function are defined by~\eqref{InCorner} and~\eqref{coef2}.
\begin{figure}[h!]
\centering
\includegraphics[height=6cm]{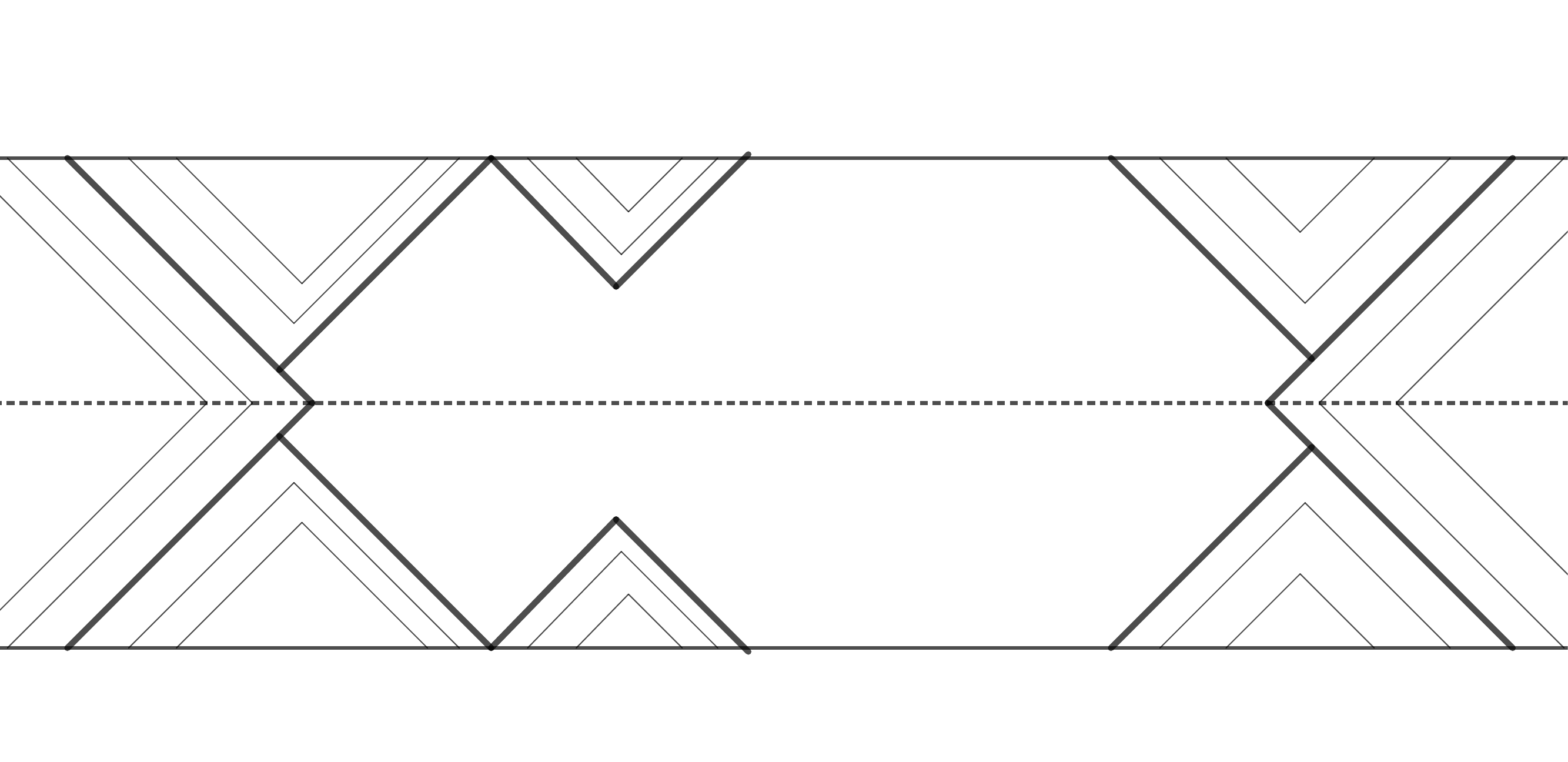}
\caption{A bilinearity domain that corresponds to a multicup}
\label{MCupFigure}
\end{figure}
One may verify that the conditions of~$C^1$-smooth concatentaion of a bilinear function in~$\BEl$ with the functions constructed by formulas~\eqref{difeq} and~\eqref{BFRT} or by~\eqref{difeq2} and~\eqref{BFLT}, are the same as~\eqref{MultiTangentR} or~\eqref{MultiTangentL} (this is completely similar to the transference of trolleybuses).

There are also multifigures that do not intersect the free boundary of~$\Omega$. We call them \emph{closed multicups}. A closed multicup is also defined by the collection of arcs~$\{\Arc_i\}_{i=1}^n$. However, since it does not intersect the free boundary, we have~$\arcr_n - \arcl_1 \leq 2$. This figure is surrounded by chordal domains from all sides. We transfer it to~$\Sigma$ in the same way as above: it matches a domain surrounded by vertical herringbones that correspond to those chordal domains. See Fig.~\ref{ClMCupFig} for visualization.

\begin{figure}[h!]
\includegraphics[height=4.5cm]{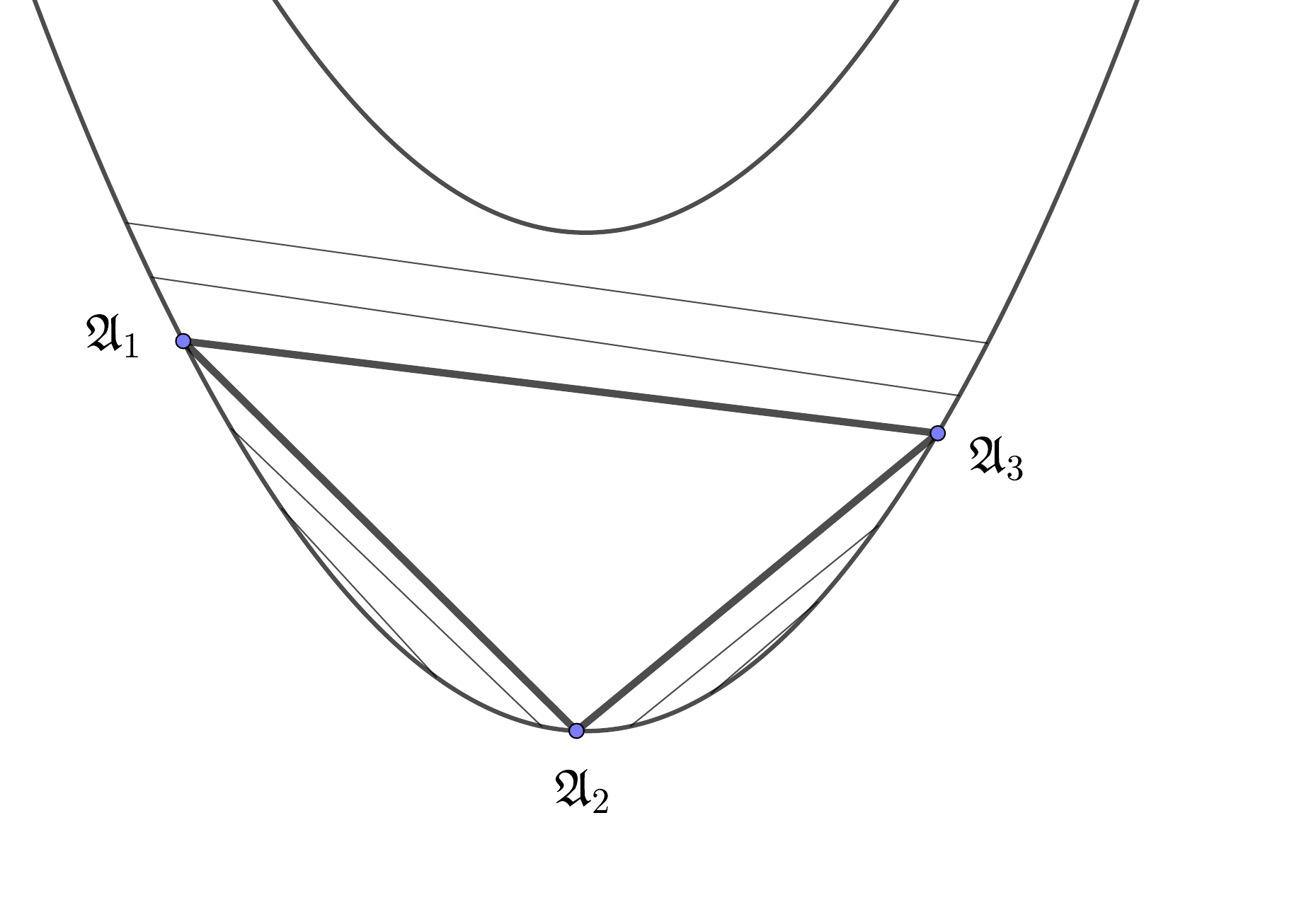}\qquad
\includegraphics[height= 5.2cm]{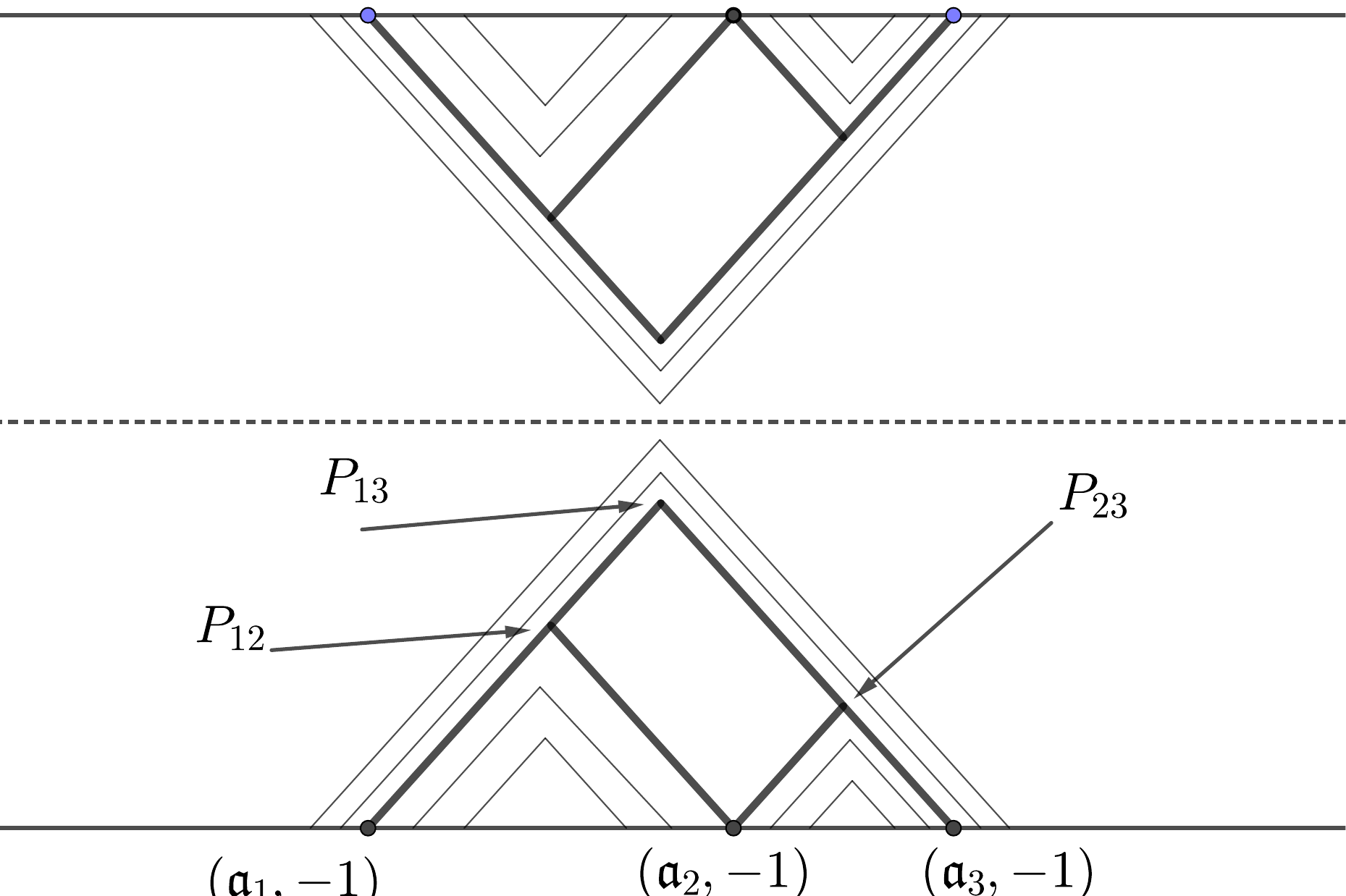}
\caption{A closed multicup and the corresponding bilinearity domain}
\label{ClMCupFig}
\end{figure}

Note that while it is considerably easy to write down the formulas for the function~$B$ in the multicup, the corresponding bilinear function~$V$ depends on its value at the {topmost point of the corresponding figure} (this is~$P_{13}$ on Fig.~\ref{ClMCupFig}). This value is computed via~\eqref{ValueTransmissionAlongVerticalFir}, and there is no short expression for it; it depends on the values of~$f$ outside~$\cup_i \Arc_i$. {In the example drawn on Fig.~\ref{ClMCupFig},} the bilinear function defines the values at the points~$P_{12}$ and~$P_{23}$. They serve as the initial data for computing the values of the function in~$\SFir([\arcr_1,\arcl_{2}],*)$ and~$\SFir([\arcr_2,\arcl_3],*)$ via~\eqref{ValueTransmissionAlongVerticalFir}. 

\begin{Rem}\label{RemarkMulti}
Since the coefficients~$\beta_0, \beta_1, \beta_2, \alpha_0, \alpha_1,$ and~$\alpha_{11}$ for multifigures are defined in the same way as in the case of a trolleybus\textup, the assertion of Theorem~\textup{\ref{TheoremForDiagonal}} that
\eq{
V(x_1,0) = B(x_1, x_1^2 + 1), \qquad (x_1,0)\in \BEl,
}
holds inside a multifigure as well.
\end{Rem}

\section{Conclusion}\label{S6}
\subsection{Combinatorial description of a general foliation}\label{s61}
Theorem~$4.4.15$ in~\cite{ISVZ2018} says that for any~$f$ satisfying Conditions~\ref{reg} and~\ref{sum} there exists a partition of~$\Omega$ into finitely many tangent domains, chordal domains, angles, trolleybuses, birdies, and multifigures such that each part falls under the scope of the corresponding proposition (e.g., chordal domains satisfy the assumptions of Proposition~\ref{LightChordalDomainCandidate}, right trolleybuses satisfy the assumptions of Proposition~\ref{RTrollProp}). This yields there exists a~$C^1$-smooth locally concave function  defined by the formulas presented in the corresponding propositions and solving the homogeneous Monge--Amp\`ere equation {almost everywhere}. It was proved in~\cite{ISVZ2018} that the constructed function coincides with the Bellman function~\eqref{BellmanFunctionInTheStrip}. The reasoning relied on optimizers. An optimizer for~$y\in \Omega$ is an optimal function~$\zeta$ for the extremal problem~\eqref{BellmanFunctionInTheStrip}. In particular, Theorem~$4.4.15$ in~\cite{ISVZ2018} delivers the desired formula for the Bellman function~\eqref{BellmanFunctionInTheStrip}. Of course, the description we have presented is far from being rigorous. We will not cite the heavy formalization from~\cite{ISVZ2018} in full and only provide several examples. The reader willing to know all the details about formalization is welcome to consult the original paper~\cite{ISVZ2018} (or a {generalization} in~\cite{ISVZ2023}).

The said paper describes a foliation with the help of a specific graph (see Subsection~$3.5.2$ in~\cite{ISVZ2018} for details). This graph has vertices that correspond to linearity domains (i.\,e., angles, trolleybuses, birdies, and multifigures) and edges that represent tangent domains and chordal domains. There are some other vertices that are needed for formalization (say, a single \emph{long} chord~$\big[(a,a^2),(b,b^2)\big]$ with~$b-a = 2$ deserves to be a vertex). A vertex of each type has numerical parameters. For linearity domains those are usually the arcs they intersect the lower boundary at. The exact list of numerical parameters may be found in the table at page~$78$ of~\cite{ISVZ2018}. A graph is called \emph{admissible} if the foliation constructed from it covers~$\Omega$ without intersections and the figure corresponding to each vertex and edge falls under the scope of the corresponding proposition. For example, all tangent domains must satisfy one of Propositions~\ref{321ISVZ},~\ref{322ISVZ},~\ref{RightTangentsCandidateInfty}, or~\ref{LeftTangentsCandidateInfty} while angles must fulfill the conditions of Proposition~\ref{344ISVZ}. The full and detailed description of the correspondence between the vertices and statements may be found in the table at page~$79$ of~\cite{ISVZ2018}.
By definition, an admissible graph allows to construct a locally concave on~$\Omega$ function with prescribed foliation satisfying the boundary conditions~\eqref{BCForBMO}. Theorem~$4.4.15$ in~\cite{ISVZ2018} claims an admissible graph exists as long as~$f$ satisfies Conditions~\ref{reg} and~\ref{sum}.

As we have seen, each linearity domain in~$\Omega$ corresponds to a bilinearity domain in~$\Sigma$, chordal domains correspond to vertical herringbones, and tangent domains correspond to horizontal herringbones~$\SRt$ or~$\SLt$. The conditions that allow the construction of a~$C^1$\!-smooth diagonally concave solution to~\eqref{StrangeEquation} in a neighborhood of each figure in~$\Sigma$ are identical to the conditions for the construction of a~$C^1$\!-smooth locally concave solution to the homogeneous Monge--Amp\`ere equation in the corresponding figure in~$\Omega$. Due to Theorem~$4.4.15$ in~\cite{ISVZ2018}, this guarantees the existence of the global diagonally concave~$C^1$\!-smooth solution to~\eqref{StrangeEquation} that falls under the scope of Theorem~\ref{MishasTh}. Therefore, in the case where~$f$ satisfies Conditions~\ref{reg} and~\ref{sum}, we have indeed constructed the minimal Bellman function~$\VB$ and our considerations provide formulas for it. In particular, we have proved Theorem~\ref{TheoremForDiagonal} for functions~$f$ that satisfy the said conditions; see Remarks~\ref{RemarkTangents},~\ref{RemarkAngle},~\ref{RemarkChords},~\ref{RemarkTrolleybus}, and~\ref{RemarkMulti}. In the forthcoming subsection, we provide an approximation argument that finishes the proof of Theorem~\ref{TheoremForDiagonal}.

 \subsection{Approximation argument}\label{s62}
Let~$f$ be a bounded from below continuous function. Without loss of generality, assume~$f$ attains positive values only. Also fix~$x_1 \in \R$. We wish to prove the identity~$\VB[f](x_1,0) = \B[f](x_1,x_1^2+1)$. By Corollary~\ref{InequalityOnDiagonal},~$\VB[f](x_1,0) \leq \B[f](x_1,x_1^2+1)$. To prove the reverse inequality, assume for a while that~$\B[f](x_1,x_1^2+1)$ is finite and pick a small number~$\eps$. By the very definition, there exists a function~$\zeta\colon [0,1]\to \R$ such that
 \eq{
\B[f](x_1,x_1^2 + 1) \leq \av{f(\zeta)}{[0,1]} + \eps,\quad \av{\zeta}{[0,1]} = x_1,\ \av{\zeta^2}{[0,1]} = x_1^2 + 1,\ \text{and}\ \|\zeta\|_{\BMO([0,1])}\leq 1.
 }
 We wish to construct a function~$\tilde{f}$ such that it satisfies Conditions~\ref{reg} and~\ref{sum}, does not exceed~$f$, and
 \eq{\label{TildefRequirement}
 \av{\tilde{f}(\zeta)}{[0,1]} \geq \av{f(\zeta)}{[0,1]} - 10\eps.
 }
 Since~$\eps$ is an arbitrary number, this will prove the desired reverse inequality:
 \mlt{
 \B[f](x_1,x_1^2 + 1) \leq \av{f(\zeta)}{[0,1]} + \eps \leq \av{\tilde{f}(\zeta)}{[0,1]} + 11\eps\\ \leq \B[\tilde{f}](x_1,x_1^2 + 1) + 11\eps = \VB[\tilde{f}](x_1,0)+ 11\eps \leq \VB[f](x_1,0)+ 11\eps.
 }
 
We will construct~$\tilde{f}$ in several steps. First, since~$f$ is non-negative and~$\B[f](x_1,x_1^2+1)$ is finite, the integral~$\int_0^1f(\zeta(t))\,dt$ converges. Therefore, there exists a large number~$N$ such that
 \eq{
 \int\limits_0^1 f(\zeta) \leq \int\limits_0^1 \min(f(\zeta),N) + \eps.
 }
Let~$f_1 = \min(f,N)$, this function is uniformly bounded and continuous. Since~$\zeta$ is a summable function, the set where~$|\zeta| > M$ has small measure when~$M$ is sufficiently large. Therefore,
\eq{
 \int\limits_0^1 f_1(\zeta) \leq \int\limits_0^1 f_2(\zeta) + \eps, \quad \text{where}\quad f_2(t) =  \min\big(f_1(t),h(|t|)\big),
}
where~$h\colon [0,\infty) \to \R_+$ is an auxiliary {smooth} function that equals~$N+1$ on~$[0,M]$ and zero on the ray~$[M+1,\infty)$.

The function~$f_2$ is bounded, continuous, and zero outside~$[-M-1,M+1]$. We approximate~$f_2$ on~$[-M-1,M+1]$ by a polynomial~$p$ such that
{
\alg{
&p(-M-1) =  p'(-M-1) = p''(-M-1) = p'''(-M-1) = 0,\\
& p(M+1)= p'(M+1) = p''(M+1) = p'''(M+1) = 0, \\
&|p(t) - f_2(t)| \leq \eps, \ t\in [-M-1,M+1].
}
}
We define
\eq{
f_3(t) = \begin{cases}
0,\quad &t\notin [-M-1,M+1];\\
p(t),\quad &t \in [-M-1,M+1].
\end{cases}
}
The function~$f_3$ fulfills Conditions~\ref{reg} and~\ref{sum}. It remains to set~$\tilde{f}(t) = f_3(t) - \eps$; this function clearly satisfies~\eqref{TildefRequirement} and does not exceed~$f$. 

The case where~$\B[f](x_1,x_1^2+1)$ is infinite is similar; in this case we construct~$\tilde{f}$ such that~$\av{\tilde{f}(\zeta)}{[0,1]}$ is arbitrarily large instead of~\eqref{TildefRequirement}.

\section{Supplementary material}\label{S7}
\subsection{Automatic local concavity}\label{s71}
We will now provide a sketch of an alternative proof of Corollary~\ref{InequalityOnDiagonal}. Recall the set~$\Lambda[f,\eps]$ given in~\eqref{SetLambda} and consider
\eq{
\tilde{\Lambda}[f,\eps] = \Set{G \colon \Omega_\eps\to \R}{G\ \text{is concave along tangents to~$\FreeBoundary \Omega_\eps$, }\ G({y_1,y_1^2}) = f(y_1),\ y_1\in \R}.
}
It is clear that if~$G \in\Lambda$, then~$G \in \tilde{\Lambda}$, but not necessarily vice versa. Define the corresponding minimal function
\eq{
\tilde{\BG}_\eps(y) = \inf\set{G(y)}{G \in \tilde{\Lambda}[f,\eps]},\qquad y\in \Omega_\eps.
}
From the above,~$\tilde{\BG}_\eps\leq \BG_\eps$. Consider the mapping~$T$
\eq{\label{Correspondence}
T(x_1,x_2) = (x_1, x_1^2+1-x_2^2),\qquad (x_1,x_2)\in \Sigma
}
and its right inverse mappings
\eq{
(y_1,y_2) \mapsto \big(y_1,\pm\sqrt{y_1^2 + 1 -y_2}\big),\qquad (y_1,y_2)\in \Omega.
}
The mapping~$T$ maps diagonal segments in~$\Sigma$ (i.\,e., segments with slopes~$\pm 1$) into tangents to~$\FreeBoundary\Omega$; the latter assertion follows from comparing formulas~\eqref{urFormula} and~\eqref{ulFormula} with the expression for~$\vr$ and~$\vl$ given in~\eqref{BFRT} and~\eqref{BFLT}. What is more, the restriction of~$T$ to any of the diagonal segments is an affine mapping.

\begin{Th}
For any function~$f$ and any~$x\in \Sigma$\textup, we have
\eq{
\VB(x) = \tilde{\BG}_1(T(x)).
}
\end{Th}

\begin{proof}
It suffices to verify that a function~$G\colon \Omega \to \R$ belongs to~$\tilde{\Lambda}[f,1]$ if and only if the composition
\eq{
G\circ T\colon x \mapsto G(x_1, x_1^2 + 1 - x_2^2)
} 
belongs to~$\Diagonal[f]$ (see~\eqref{DiagonalClass}). In fact, this is a direct consequence of the properties of~$T$: this map takes diagonal segments to tangents, and the restriction of~$T$ to any diagonal segment is affine. It remains to use that the concavity is preserved by affine maps.
\end{proof}
Note that together with the inequality~$\tilde{\BG}_1\leq \BG_1$, the theorem above gives another proof of Corollary~\ref{InequalityOnDiagonal}. Theorem~\ref{TheoremForDiagonal} says~$\tilde{\BG}_1(y_1,y_1^2 + 1) = \BG_1(y_1,y_1^2+1)$. We believe there exists a simpler proof of this fact that does not {rely on consideration of the} types of foliations.

Consider the class of \emph{subtangential} segments
\eq{
\NT_\eps = \Set{\ell \subset \Omega_\eps}{\ell \ \text{is a segment such that the line containing it does not intersect}\ \FreeBoundary\Omega_\eps}
}
and the functions that are concave along these segments
\eq{\label{Non-tangD}
\Lambda^*[f,\eps] = \Set{G\colon \Omega_\eps\to \R}{\forall \ell \in \NT_\eps \quad G|_\ell \ \text{is concave and}\ G({y_1,y_1^2}) = f(y_1),\ y_1 \in \R}.
}
We define the third minimal function
\eq{\label{Non-tangF}
\BG^*_\eps[f,\eps](y) =\inf\Set{G(y)}{G\in \Lambda^*[f,\eps]},\qquad y\in \Omega_\eps. 
}
Clearly,~$\tilde{\BG}_\eps \leq \BG^*_\eps \leq \BG_\eps$, and on the free and fixed boundaries, these three functions coincide. 

\begin{Fact}\label{Auto}
For any continuous uniformly bounded from below function~$f$\textup, we have~$\BG_\eps = \BG^*_\eps$.
\end{Fact} 

The proof of this fact is similar to that of Theorem~\ref{TheoremForDiagonal}: we study the foliation of~$\BG_\eps$ described in~\cite{ISVZ2018} and see that the extremal segments are subtangential. We omit the details, because we believe there must be a simpler geometric proof that avoids considering many cases of various foliations. Note that Fact~\ref{Auto} is a statement about automatic concavity/convexity: a somehow weakly concave function is proved to fulfill a stronger convexity condition under additional assumptions. The most famous statement of this type is the Morrey conjecture (see the recent preprint~\cite{AFGKK2023} for the survey on the Morrey conjecture and related topics). See~\cite{KirchheimKristensen2016} for applications of a similar phenomenon for separately convex functions to variational problems and inequalities for differential operators. 

A passage to domains more general than~$\Omega_\eps$ may help to find a more geometric proof of Fact~\ref{Auto} and, maybe, Theorem~\ref{TheoremForDiagonal}. Using theory from~\cite{ISVZ2023}, one may generalize Fact~\ref{Auto} to planar domains that are set-theoretical differences of two unbounded convex sets (the precise description of the class of available sets is long and more restrictive). A ring is the simplest set that is not covered by the mentioned theory:
\eq{
\set{y \in \R^2}{1 - \delta \leq y_1^2 + y_2^2 \leq 1},\qquad\text{where}\quad \delta \in (0,1)
} 
is an arbitrary number. This set is related to~$S^1$-valued~$\BMO$ the same way as~$\Omega_\eps$ is related to the classical~$\BMO$, see~\cite{SZ2022}. In this case,~$S^1 = \set{y\in \R^2}{y_1^2 + y_2^2 = 1}$ is the fixed boundary of the ring; the smaller circle is the free boundary. For any function~$f\colon S^1 \to \R$ we may consider the minimal locally concave function~$\BG^\circ_\delta[f]$ defined on the ring. We may also consider the subtangentially locally concave function~$\BG^{\circ,*}_\delta[f]$ defined similar to~\eqref{Non-tangD} and~\eqref{Non-tangF}.
\begin{Conj}\label{Conjecture713}
Let~$f\colon S^1 \to \R$ be a continuous function. Then~$\BG^\circ_\delta[f] = \BG^{\circ,*}_\delta[f]$ for any~$\delta > 0$.
\end{Conj}
\subsection{On smoothness}\label{s72}
The minimal functions we study are concave  {(almost everywhere)} solutions to the homogeneous Monge--Amp\`ere equation~\eqref{MongeAmpere} and to~\eqref{StrangeEquation}. Though for our applications, the formulas giving analytic expressions for solutions are of primary importance, it is also interesting to tackle these problems from a PDE point of view and study the smoothness properties of solutions. Seemingly, the phenomenon of free boundary we discuss in this paper has not appeared in PDE literature. However,  see~\cite{Krylov1989} for a very general setting, where the boundary data is defined on a part of the boundary; seemingly, the notion of `free boundary' of the latter paper is different from ours. What is more, seemingly, the equations we study do not fulfill the non-degeneracy conditions of the latter paper. The homogeneous Monge--Amp\`ere equation~\eqref{MongeAmpere} on a strictly convex bounded domain in~$\R^d$ was studied in~\cite{CNS1986}. It was proved that the solution is of class~$C^{1,1}$, provided the boundary data and the boundary itself are~$C^{3,1}$-smooth. Recall the notation~$C^{k,\alpha}$,~$k\in \mathbb{N}\cup \{0\}$ and~$\alpha \in (0,1]$ means the function is~$k$ times differentiable and its~$k$-th differential is~$\alpha$-H\"older continuous. The same paper also provides an example showing the sharpness of this result: there exists a smooth function on the boundary of the unit circle, for which the solution is not of class~$C^2$. {In~\cite{Guan1998}, similar results are obtained on non-convex domains with smooth boundary; the boundary values are fixed on the whole boundary (there is no free part of the boundary as well).} The paper~\cite{LW2015} considers the homogeneous Monge--Amp\`ere boundary value problem on a strip~$\set{x\in \R^d}{0\leq x_d \leq 1}$ and provides the if and only if conditions on concave boundary data for the solvability. It is also proved that in the situation of convex boundary data, the regularity self-improves. Roughly speaking, the solution is of the same regularity class as the boundary values. In some cases, a similar effect is present in the case of bounded domain considered in~\cite{CNS1986}, see the last chapter of~\cite{LW2015}.

Motivated by the above, we may state two conjectures. Maybe, the assumptions on~$f$ are too strong.
\begin{Conj}\label{Smoothness1}
Let~$f$ be a smooth uniformly bounded function. Then\textup, $\B_\eps$ is of class~$C^{1,1}$.
\end{Conj}
\begin{Conj}\label{Smoothness2}
Let~$f$ be a smooth uniformly bounded function. Then\textup, the functions~$\UB$ and~$\VB$ are of class~$C^{1,1}$.
\end{Conj}

\subsection{Discretization and modeling}\label{s73}
Theorem~\ref{TheoremForDiagonal} allows to compute approximations to the trace of~$\B$ on the upper parabola~$(x_1,x_1^2 + 1)$. Namely, the extremal problem~\eqref{MinimalDiagonallyConcaveFunction} for the function~$\VB$ allows efficient discretization, while it is not clear whether one may find good discrete approximations of~\eqref{MinimalLocallyCOncave}. Pick a large number~$N$ and consider the rectangular lattice~$N^{-1}(\Z\times \Z)$. Pick another large natural number~$M$ and consider the part of the lattice
\eq{
L_{N,M} = \Set{(m,n)\in \Z\times \Z}{|n| \leq N, |m|\leq MN}.
}
We say that a function~$G \colon L_{N,M}\to \R$ is diagonally concave, provided
\alg{
G(m,n) &\geq \frac12\big(G(m+1,n+1) + G(m-1,n-1)\big) \ \text{and}\\
G(m,n) &\geq \frac12\big(G(m+1,n-1) + G(m-1,n+1)\big),
}
whenever the arguments lie in~$L_{N,M}$. For any boundary data~$f$, we may pose the extremal problem of finding the minimal diagonally concave function~$\VB_{N,M}\colon L_{N,M} \to \R$ satisfying the boundary condition
\eq{
\VB_{N,M}(m,\pm N) = f(m/N).
}
\begin{Rem}
The function~$(m,n)\to \VB(m/N, n/N)$ is diagonally concave on~$L_{M,N}$ and satisfies the boundary condition. Thus,
\eq{
\VB_{N,M}(m,n) \leq \VB(m/N, n/N) \qquad \text{for all}\quad (m,n)\in L_{M,N}.
}
\end{Rem}
One may prove that~$\VB_{N,M}(m,n)$ tends to~$ \VB(m/N, n/N)$ when~$M$ and~$N$ tend to infinity, and also estimate the discrepancy in terms of various smoothness characteristics of~$f$. What is so nice in these approximations, is that they are very easy to compute (also approximatively). Define the function~$G_0\colon L_{N,M}\to \R$ by the rule
\eq{
G_0(m,n) = \begin{cases}
\min\limits_{[-M,M]} f(x), \quad |n|< N;\\
f\big(m/N\big),\quad |n| = N.
\end{cases}
}
Then, define the functions~$G_1,G_2,\ldots, G_k, \ldots$ recursively:
\mlt{
G_{k+1}(m,n) = \max\Big(G_{k}(m,n), \frac12\big(G_k(m+1,n+1) + G_k(m-1,n-1)\big),\\ \frac12\big(G_k(m+1,n+1) + G_k(m-1,n-1)\big)\Big).
}
The functional sequence~$G_{k}$ does not decrease and converges to~$\VB_{N,M}$. One may take~$G_k$ with sufficiently large~$k$ as an approximation to~$\VB_{N,M}$. On Fig.~\ref{Res}, we exhibit the results of these experiments for the case~$f(t) = t^5/60 - ct^3/6$; see the detailed description of the function~$\B$ for this boundary data in Section~$6.6$ of~\cite{IOSVZ2016}.
\begin{figure}[h!]
\includegraphics[height=1.6cm]{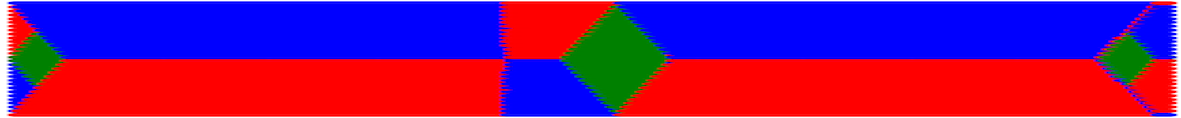}
\includegraphics[height=1.6cm]{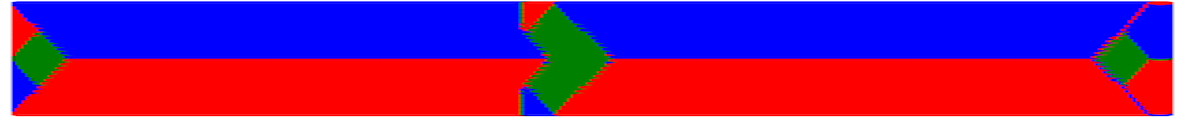}
\includegraphics[height=1.6cm]{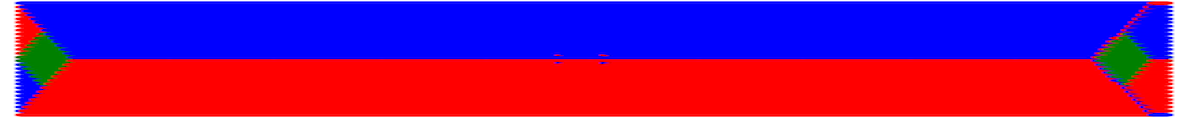}
\caption{Results of the experiment}
\label{Res}
\end{figure}
We have chosen~$N=30$,~$M=10$, and~$c=2$, $1.1$, and~$0.5$, respectively. The paper~\cite{IOSVZ2016} suggests that for~$c \leq 1$ the whole parabolic strip~$\Omega$ is foliated by left tangents, for~$c \in (1, \frac{1614}{1225})$, there will be a trolleybus, and for~$c > \frac{1614}{1225}$, there will be a cup and an angle. The three examples confirm these results with a slight shift in the parameter~$c$, which, seemingly, occurs from the truncation of the domain. Let us explain what is drawn on Fig.~\ref{Res} in more detail.

We compute the function~$\VB_{N,M}$ with the accuracy~$10^{-5}$ in the sense that we replace~$\VB_{N,M}$ by~$G_k$, where~$k$ is the first moment such that~$\|G_{k} - G_{k-1}\|_{\ell_\infty} \leq 10^{-5}$. Then, we examine the `flatness' of the constructed function. We compute two quantities
\alg{
D_1(m,n) &= \Big|G(m,n) - \frac{G(m+1,n+1)+ G(m-1,n-1)}{2}\Big|;\\
D_2(m,n) &= \Big|G(m,n) - \frac{G(m+1,n-1)+ G(m-1,n+1)}{2}\Big|
}
for any point~$(m,n) \in L_{M,N}$ (except for the boundary). If both these quantities are smaller than~$0.005$, then~$(m,n)$ is a `bilinearity point' and we mark it with green color. In the case~$D_1 < 0.005$ and~$D_2 > 0.005$, we mark~$(m,n)$ with red; in the case of reverse inequalities, we mark~$(m,n)$ with blue. Of course, there are some uncolored points~$(m,n)$ (for example, points on the boundary of~$L_{M,N}$). The choice of the numbers~$10^{-5}$ and~$0.005$ is close to being arbitrary, there are two reasonable limitations: the former number should be significantly smaller than the latter one, and the latter one should be significantly smaller than the `average' value of the derivative of the boundary data. It takes a regular laptop about 10 minutes to make the computation and draw the picture; of course, the algorithm is far from being optimal.

Thus, on the first picture in Fig.~\ref{Res}, we see a square, a vertical herringbone, and two left and one right horizontal herringbones separating and surrounding them. We disregard the drawing near the boundaries~$|y_1| = \pm M$. On the second picture, we see a corner surrounded by left horizontal herringbones. In the last picture, there is a single left horizontal herringbone. This matches the scenario predicted by the results of the present paper and Section~$6.6$ in~\cite{IOSVZ2016} (modulo a shift in the parameter~$d$). The experiments we have described here were, seemingly, known already to Burkholder, because on p. 654 of~\cite{Burkholder1984} he mentions that the experiments helped him to guess the minimal diagonally concave function. 

{\subsection{List of conjectures and questions}\label{s74}
\begin{enumerate}
\item What is the sharp form of Corollary~\ref{BoundaryBehavior}? Can one find simple if and only if conditions for~$f$ that make the functions~$\B$,~$\VB$, and~$\UB$ finite?

\item The continuity of~$f$ in Theorem~\ref{TheoremForDiagonal} seems redundant. Does~\eqref{TheoremForDiagonalFormula} hold for~$f$ being merely measurable, locally bounded, and globally bounded from below? The theory of~\cite{StolyarovZatitskiy2016} works with such functions.

\item Question~\ref{Que1} asks when the inequality in Theorem~\ref{TheoremForSubordination} turns into equality. A related (but slightly different question): when are~$\UB$ and~$\VB$ equal? The function~$f(t) = e^{\lambda t}$ provides us with the example of coincidence of the two functions (see Remark~\ref{U=V}), whether in the case~$f(t) = -e^{\lambda t}$ we have~$\VB < \UB$ (see the discussion after the cited remark).

\item Find a simple proof of Fact~\ref{Auto} or prove Conjecture~\ref{Conjecture713}.

\item Prove Conjecture~\ref{Smoothness1} or Conjecture~\ref{Smoothness2}.

\item The classical theory of Burkholder functions is generalized to the setting of vector-valued martingales, see~\cite{Osekowski2012}; some of these functions are simpler in the vectorial setting. What happens to the functions~$\UB$ and~$\VB$ defined in~\eqref{UDef} and~\eqref{VDef} if the martingales~$\varphi$ and~$\psi$ attain their values in a Hilbert space?
\end{enumerate}
}
\bibliography{mybib}{}

\providecommand{\bysame}{\leavevmode\hbox to3em{\hrulefill}\thinspace}
\providecommand{\MR}{\relax\ifhmode\unskip\space\fi MR }
\providecommand{\MRhref}[2]{%
  \href{http://www.ams.org/mathscinet-getitem?mr=#1}{#2}
}
\providecommand{\href}[2]{#2}
\begin{thebibliography}{10}

\bibitem{AFGKK2023}
K.~Astala, D.~Faraco, A.~Guerra, A.~Koski, and J.~Kristensen, \emph{The local
  {B}urkholder functional, quasiconvexity and {G}eometric {F}unction {T}heory},
  https://arxiv.org/abs/2309.03495.

\bibitem{Boole1857}
G.~Boole, \emph{On the comparison of transcendents, with certain applications
  to definite integrals}, Philos. Trans. R. Soc. Lond. Ser. \textbf{147}
  (1857), 745--803.

\bibitem{Burkholder1984}
D.~L. Burkholder, \emph{Boundary value problems and sharp inequalities for
  martingale transforms}, Ann. Prob. \textbf{12} (1984), no.~3, 647--702.

\bibitem{Burkholder1988}
\bysame, \emph{Sharp inequalities for martingales and stochastic integrals},
  no. 157--158, 1988, Colloque Paul L\'{e}vy sur les Processus Stochastiques
  (Palaiseau, 1987), pp.~75--94.

\bibitem{Burkholder1991}
\bysame, \emph{Explorations in martingale theory and its applications},
  \'{E}cole d'\'{E}t\'{e} de {P}robabilit\'{e}s de {S}aint-{F}lour
  {XIX}---1989, Lecture Notes in Math., vol. 1464, Springer, Berlin, 1991,
  pp.~1--66.

\bibitem{CNS1986}
L.~Caffarelli, L.~Nirenberg, and J.~Spruck, \emph{The {D}irichlet problem for
  the degenerate {M}onge--{A}mp\`ere equation}, Rev. Mat. Iberoamericana
  \textbf{2} (1986), no.~1--2, 19--27. \MR{864651}

\bibitem{CLM2010}
L.~Colzani, E.~Laeng, and L.~Monz\'{o}n, \emph{Variations on a theme of {B}oole
  and {S}tein--{W}eiss}, J. Math. Anal. Appl. \textbf{363} (2010), no.~1,
  225--229. \MR{2559057}

\bibitem{Guan1998}
B.~Guan, \emph{The {D}irichlet problem for {M}onge-{A}mp\`ere equations in
  non-convex domains and spacelike hypersurfaces of constant {G}auss
  curvature}, Trans. Amer. Math. Soc. \textbf{350} (1998), no.~12, 4955--4971.

\bibitem{IOSVZ2016}
P.~Ivanishvili, N.~N. Osipov, D.~M. Stolyarov, V.~I. Vasyunin, and P.~B.
  Zatitskiy, \emph{Bellman function for extremal problems in~$\mathrm{BMO}$},
  Trans. Amer. Math. Soc. \textbf{368} (2016), 3415--3468.

\bibitem{ISVZ2023}
P.~Ivanishvili, D.~M. Stolyarov, V.~I. Vasyunin, and P.~B. Zatitskii,
  \emph{Bellman functions on simple non-convex domains in the plane},
  https://arxiv.org/abs/2305.03523.

\bibitem{ISVZ2018}
P.~Ivanishvili, D.~M. Stolyarov, V.~I. Vasyunin, and P.~B. Zatitskiy,
  \emph{Bellman function for extremal problems on~$\mathrm{BMO}$ \textup{II:}
  evolution}, Mem. Amer. Math. Soc. \textbf{255} (2018), no.~1220.

\bibitem{ISZ2015}
P.~Ivanisvili, D.~M. Stolyarov, and P.~B. Zatitskii, \emph{Bellman vs
  {B}eurling: sharp estimates of uniform convexity for {$L^p$} spaces}, Algebra
  i Analiz \textbf{27} (2015), no.~2, 218--231. \MR{3444467}

\bibitem{Kemperman1973}
J.~H.~B. Kemperman, \emph{The general moment problem, a geometric approach},
  Ann. Math. Statist. \textbf{39} (1968), 93--122. \MR{247645}

\bibitem{KirchheimKristensen2016}
B.~Kirchheim and J.~Kristensen, \emph{On rank one convex functions that are
  homogeneous of degree one}, Archive for Rational Mechanics and Analysis
  \textbf{221} (2016), no.~1, 527--558.

\bibitem{KN1977}
M.~G. Kre\u{\i}n and A.~A. Nudelman, \emph{The {M}arkov moment problem and
  extremal problems}, Translations of Mathematical Monographs, Vol. 50,
  American Mathematical Society, Providence, R.I., 1977, Ideas and problems of
  P. L. \v{C}eby\v{s}ev and A. A. Markov and their further development,
  Translated from the Russian by D. Louvish. \MR{0458081}

\bibitem{Krylov1989}
N.~V. Krylov, \emph{Smoothness of the payoff function for a controllable
  diffusion process in a domain}, Izv. Akad. Nauk SSSR Ser. Mat. \textbf{53}
  (1989), no.~1, 66--96. \MR{992979}

\bibitem{LW2015}
Q.-R. Li and X.-J. Wang, \emph{Regularity of the homogeneous
  {M}onge--{A}mp\`ere equation}, Disrcete and continuous dynamical systems
  \textbf{35} (2015), no.~12, 6069--6084.

\bibitem{Novikov2022}
M.~I. Novikov, \emph{Sufficient conditions for the minimality of biconcave
  functions}, Alg. i Anal. \textbf{34} (2022), no.~5, 173--210, (in Russian);
  to be translated in St. Petersburg Math. J.

\bibitem{Osekowski2012}
A.~Os\c{e}kowski, \emph{Sharp martingale and semimartingale inequalities},
  Monografie Matematyczne IMPAN~{\bf 72}, Springer Basel, 2012.

\bibitem{RT1977}
J.~Rauch and B.~A. Taylor, \emph{The {D}irichlet problem for the
  multi-dimensional {M}onge--{A}mp\`ere equation}, Rocky Mountain J. Math.
  \textbf{7} (1977), 345--364.

\bibitem{Shiryaev2019}
A.~N. Shiryaev, \emph{Probability. 2}, Graduate Texts in Mathematics, vol.~95,
  Springer, New York, 2019, Third edition of [MR0737192], Translated from the
  2007 fourth Russian edition by R. P. Boas and D. M. Chibisov.

\bibitem{SlavinVasyunin2011}
L.~Slavin and V.~Vasyunin, \emph{Sharp results in the integral-form
  {J}ohn-{N}irenberg inequality}, Trans. Amer. Math. Soc. \textbf{363} (2011),
  no.~8, 4135--4169. \MR{2792983}

\bibitem{SlavinVasyunin2012}
Leonid Slavin and Vasily Vasyunin, \emph{Sharp {$L^p$} estimates on {BMO}},
  Indiana Univ. Math. J. \textbf{61} (2012), no.~3, 1051--1110. \MR{3071693}

\bibitem{Stein1993}
E.~M. Stein, \emph{Harmonic analysis: real-variable methods, orthogonality, and
  oscillatory integrals}, Princeton University Press, 1993.

\bibitem{SteinWeiss1959}
E.~M. Stein and G.~Weiss, \emph{An extension of a theorem of {M}arcinkiewicz
  and some of its applications}, J. Math. Mech. \textbf{8} (1959), 263--284.
  \MR{0107163}

\bibitem{StolyarovZatitskiy2016}
D.~M. Stolyarov and P.~B. Zatitskiy, \emph{Theory of locally concave functions
  and its applications to sharp estimates of integral functionals}, Adv. Math.
  \textbf{291} (2016), 228--273.

\bibitem{SVZZ2022}
Dmitriy Stolyarov, Vasily Vasyunin, Pavel Zatitskiy, and Ilya Zlotnikov,
  \emph{Sharp moment estimates for martingales with uniformly bounded square
  functions}, Math. Z. \textbf{302} (2022), no.~1, 181--217. \MR{4462673}

\bibitem{SZ2022}
P.~B. Zatitskiy and D.~M. Stolyarov, \emph{On locally concave functions on
  simplest nonconvex domain}, Zap. Nauchn. Sem. S.-Peterburg Otdel. Mat. Inst.
  Steklov. (POMI) \textbf{512} (2022), no.~Issledovaniya po Line\u{\i}nym
  Operatoram i Teorii Funktsi\u{\i}. 50, 40--87. \MR{4508359}

\end{thebibliography}
\bibliographystyle{amsplain}

\bigskip

Dmitriy Stolyarov

St. Petersburg State University,

d.m.stolyarov at spbu dot ru

\bigskip

Vasily Vasyunin

St. Petersburg State University,

vasyunin at pdmi dot ras do ru

\bigskip

Pavel Zatitskii

University of Cincinnati,

St. Petersburg State University,

pavelz at pdmi dot ras dot ru

\end{document}